\documentclass[11pt,a4paper]{amsart}
\usepackage{latexsym,amsfonts,amsthm,amsmath,mathrsfs,amssymb}
\usepackage[dvips]{color}
\usepackage{pdfsync}
\usepackage[dvips,final]{graphics}
\usepackage{epstopdf}
\usepackage{pstricks, pst-node, pst-text, pst-3d}

\theoremstyle{plain}
\newtheorem{theorem}{Theorem}[section]
\newtheorem{lem}[theorem]{Lemma}
\newtheorem{prop}[theorem]{Proposition}
\newtheorem{coroll}[theorem]{Corollary}

\theoremstyle{definition}
\newtheorem{defi}{Definition}[section]

\theoremstyle{remark}
\newtheorem{rem}[theorem]{Remark}

\newcommand{\R}{\mathbb{R}}
\newcommand{\N}{\mathbb{N}}
\newcommand{\M}{\mathbb{M}}

\newcommand{\Q}{\mathcal{Q}}
\newcommand{\W}{\mathbb{W}}
\newcommand{\V}{\mathbb{V}}
\newcommand{\T}{\mathbb{T}}
\newcommand{\C}{\mathbb{C}}
\newcommand{\E}{\mathcal{E}}
\newcommand{\HH}{\mathbb{H}}
\newcommand{\B}{\mathcal{B}}
\newcommand{\G}{\mathbb{G}}

\newcommand{\eps}{\epsilon}

\newcommand{\graph}[1]{\mathrm{graph}\,(#1)}
\newcommand{\norm}[1]{\left\Vert{#1}\right\Vert}

\newcommand{\mfrak}{\mathfrak}
\newcommand{\mcal}{\mathcal}

\newcommand{\average}{{\mathchoice {\kern1ex\vcenter{\hrule height.4pt
width 6pt
depth0pt} \kern-9.7pt} {\kern1ex\vcenter{\hrule height.4pt width 4.3pt
depth0pt}
\kern-7pt} {} {} }}
\newcommand{\ave}{\average\int}

\title{Intrinsic differentiability and Intrinsic Regular Surfaces in Carnot groups}

\date{\today}
\author[D. Di Donato]{Daniela Di Donato}
\address{Daniela Di Donato: Dipartimento di Matematica\\Universit\`a di Trento\\ Via Sommarive 14\\ 38123, Povo (Trento) - Italy\\} 
\email{daniela.didonato@unitn.it}
    \thanks{D.D.D. is supported by University of Trento, Italy.}


\begin{document}

\begin{abstract}
A Carnot group $\G$ is a connected, simply connected, nilpotent Lie group with stra\-ti\-fied Lie algebra. Intrinsic regular surfaces in Carnot groups play the same role as $\C^1$ surfaces in Euclidean spaces. As in Euclidean spaces, intrinsic regular surfaces can be locally defined in different ways: e.g. as non critical level sets or as continuously intrinsic differentiable graphs. The equivalence of these natural definitions is the problem that we are studying. Precisely our aim is to generalize the results on \cite{biblio1} valid in Heisenberg groups to the more general setting of Carnot groups.

\end{abstract}

\maketitle

\section{Introduction}

 In the last years a systematic attempt to develop a good notion of rectifiable sets in metric space and in particular inside Carnot groups, has become the object of many studies. For a general theory of rectifiable sets in Euclidean spaces one can see  \cite{biblio5, biblio4, biblio16} while a general theory in metric spaces can be found in \cite{biblioAMBkirc}. 

Rectifiable sets are classically defined as contained in  the countable union of $\C^1$ submanifolds. In this paper we focus our attention on the natural notion of $\C^1$ surface, inside a special class of metric spaces i.e. the Carnot groups $\G$ of step $\kappa$. A short description of Carnot groups is in Section 2. Here we simply recall that they are connected, simply connected Lie group whose Lie algebra $\mathfrak g$ admits a step $\kappa$ stratification. Through the exponential map, a Carnot group $\G$ can be identified with $\R^N$, for a certain $N>0$, endowed with a non commutative  group operation. 

Euclidean spaces are commutative Carnot groups and are  the only commutative ones. The simplest but, at the same time, non-trivial instances of  non-Abelian Carnot groups are provided by the Heisenberg groups $\mathbb{H}^n$ (see for instance \cite{biblio3}).

A Carnot group $\G$ is endowed with a natural left-invariant metric $d$. 
Non commutative Carnot groups, endowed with their left invariant metric are not Riemannian manifolds  not even locally. In fact they are particular instances of so called sub Riemannian manifolds.


Main objects of study in this paper are the notions of regular surfaces and of intrinsic graphs and their link.

Intrinsic regular surfaces in Carnot groups should play the same role as $\C^1$ surfaces in Euclidean spaces. In Euclidean spaces, $\C^1$ surfaces can be locally defined in different ways: e.g. as non critical level sets of $\C^1$ functions or, equivalently, as graphs of $\C^1$ maps between complementary linear subspaces. In Carnot groups the equivalence of these definitions is not  true any more. One of the main aim of this paper is to find the additional assumptions in order that these notions are equivalent in $\G$. Precisely we want to generalize the results on \cite{biblio1} valid in Heisenberg groups to the more general setting of Carnot groups.

Here by the word $intrinsic$ and $regular$ we want to emphasize a privileged role played by group translations and dilations, and its differential structure as Carnot-Carath\'eodory manifold in a sense we will precise below. 

We begin recalling that an intrinsic regular hypersurface (i.e. a topological codimension 1 surface) $S\subset \G$ is locally defined as a non critical level set of a $\C^1$ intrinsic function. More precisely, there exists a continuous function $f:\G \to \R$ such that locally $S=\{ P\in \G : f(P)=0\}$ and the intrinsic gradient $\nabla _\G f=(X_1f,\dots , X_m f)$ exists in the sense of distributions and it is continuous and non vanishing on $S$. In a similar way, a $k$-codimensional regular surface $S\subset \G$ is locally defined as a non critical level set of a $\C^1$ intrinsic vector function $F:\G\to \R^k$.


On the other hand, the intrinsic graphs came out naturally in \cite{biblio6}, while studying level sets of Pansu differentiable functions from $\mathbb{H}^n$ to $\R$.  The simple idea of intrinsic graph is the following one: let $\mathbb{M}$ and $\W$ be complementary subgroups of $\G$, i.e. homogeneous subgroups such that  $\W \cap \mathbb{M}= \{ 0 \}$ and $\G=\W\cdot \mathbb{M}$ (here $\cdot $ indicates the group operation in $\G$ and $0$ is the unit element), then the intrinsic left graph of $\phi :\W\to \mathbb{M}$ is the set
\begin{equation*}
\graph {\phi}:=\{ A\cdot \phi (A) \, |\, A\in \W \}.
\end{equation*}
Hence the existence of intrinsic graphs depends on the possibility of splitting $\G$ as a product
of complementary subgroups hence it depends on the structure of the algebra $\mathfrak g$.


 By Implicit Function Theorem, proved in \cite{biblio6} for the Heisenberg group and in  \cite{biblio7} for a general Carnot group (see also Theorem 1.3, \cite{biblioMAGNANI}) it follows
 \begin{equation*}
\mbox{a }  \G \mbox{-regular surface $S$ locally is an intrinsic graph of a suitable function } \phi.
\end{equation*}

Consequently, given an intrinsic graph $S=\graph {\phi} \subset \G$, the main aim of this paper is to find necessary and sufficient assumptions on $\phi $ in order that the opposite implication is true. 



Following \cite{biblioLIBROSC}, in Section 3 we define an appropriate notion of differentiability  for a map acting between complementary subgroups of $\G$  here denoted as \textit{uniform intrinsic  differentiability} (see Definition $\ref{Ndef1.1}$).  In Theorem \ref{teo4.1}  we characterize  $\G$-regular intrinsic graphs as graphs of uniformly intrinsic  differentiable functions $\phi:\E \subset \W \to \M$ where $\G$ is a step $\kappa$ Carnot group, $\W, \M$ complementary subgroups, with $\M$  horizontal and $k$ dimensional. This result  generalizes Theorem 1.2 in \cite{biblio1} proved in Heisenberg groups (see also \cite{biblio12}).
As a  corollary of this result and Theorem 3.3.5 in \cite{biblioKOZHEVNIKOV}, 
 we get a comparison between the Reifenberg vanishing flat set  and the uniformly intrinsic differentiable map (see Theorem $\ref{biblioKOZHEVNIKOV}$).

When $\M$ is one dimensional we can identify $\phi:\E \subset \W \to \M$ with  a real valued continuous function defined on an one codimensional homogeneous subgroup of $\G$ (see Remark \ref{remIMPORT}). In this case in Heisenberg groups, it is known after the results in \cite{biblio1, biblio27} that the intrinsic differentiability of $\phi$ is equivalent to the existence and continuity of suitable \lq derivatives\rq\,  $D^\phi _j\phi$ of $\phi$.  The non linear first order differential operators $D^\phi _j$ were introduced by Serra Cassano et al. in the context of Heisenberg groups $\HH^n$ (see \cite{biblioLIBROSC} and the references therein).  Following the notations in \cite{biblioLIBROSC} the operators $D^\phi _j$ are denoted as \emph{intrinsic derivatives} of $\phi$ and $D^\phi \phi$, the vector of the intrinsic derivatives of $\phi$, is the  intrinsic gradient of $\phi$.  
 In the first Heisenberg group $\mathbb{H}^1$  the intrinsic derivative  $D^\phi \phi$ reduces to the classical Burgers' equation. 
 
  In \cite{biblio1, biblio17, biblio27}, the authors introduce and investigate some suitable notions of weak solution for the non-linear first order PDEs' system
\begin{equation}\label{PINCO}
D^{\phi } \phi  =w \quad \mbox{ in} \,\, \mcal O,
\end{equation}
being $w$ a prescribed continuous function and $\mcal O \subset \R^{N-1}$. 
In 
\cite{biblio1} and \cite{biblio27} $\phi$ and $w$ are continuous functions, while in \cite{biblio17} $w$ is only a bounded measurable function.



In particular in \cite{biblio1} it was introduced the concept of broad* solution of the system $\eqref{PINCO}$ (see Definition $\ref{defbroad*}$). In $\HH^1$ this notion extends the classical notion of broad solution for Burgers' equation through characteristic curves provided  $\phi$ and $w$ are locally Lipschitz continuous. In our case  $\phi$ and $w$  are supposed to be only continuous then the classical theory breaks down. On the other hand broad* solutions of the system $\eqref{PINCO}$ can be constructed with a continuous datum $w$.

In Section 5  we extend the results in \cite{biblio1} using the  notion of  broad* solution of the system $\eqref{PINCO}$. We study real valued functions defined on an one codimensional homogeneous subgroup of a Carnot group of step 2. We define the appropriate notion of intrinsic derivative in this setting and we extend Theorem 1.3 and Theorem 5.7 in \cite{biblio1} proved in $\HH^n$.
Indeed we analyze the $\G$-regular hypersurfaces in a subclass of step two Carnot groups including the Heisenberg groups. This class is shown in \cite{biblio3} (see Section 5.1).
 Precisely, in Theorem \ref{finale1} we prove that the intrinsic graph of continuous map $\phi$ is  a regular surface if and only if $\phi$ is broad* solution of $\eqref{PINCO}$ and it is 1/2-little H\"older continuous (see Definition \ref{big3.3.11}). We also show that these assumptions are equivalent to the fact that $\phi $ and its intrinsic gradient can be uniformly approximated by  $\C^1$ functions. 
 \\
 
$\mathbf{Acknowledgements.}$ We wish to express our gratitude to R.Serapioni and F.Serra Cassano, for having signaled us this problem and for many invaluable discussions during our PhD at University of Trento. We thank A.Pinamonti for important suggestions on the subject. We also thank B.Franchi e D.Vittone for a careful reading of our PhD thesis and of this paper.



\section{Notations and preliminary results}

\subsection{Carnot groups} We begin by recalling briefly the definition of Carnot groups. For a general account see e.g. \cite{biblio3, biblio22, biblioLeDonne, biblioLIBROSC}. 

A Carnot group $\G=(\G, \cdot, \delta_\lambda)$ of step $\kappa$  is a connected and simply connected Lie group whose Lie algebra $\mathfrak g$ admits a stratification, i.e. a direct sum decomposition $\mathfrak g=V_1\,  \oplus \, V_2 \, \oplus \dots \oplus \, V_\kappa$. The stratification has the further property  that the entire Lie algebra $\mathfrak g$ is generated by its first layer $V_1$, the so called horizontal layer, that is 
\[
\left\{
\begin{array}{l}
[V_1, V_{i-1}]= V_i \quad \mbox{if } \, 2\leq i \leq \kappa \\ \hspace{0,05 cm}  [V_1,V_\kappa]=\{0 \}
\end{array}
\right.
\]
We denote by $N$ the dimension of $\mathfrak g$ and by $n_s$ the dimension of $V_s$.  

The exponential map $\exp :\mathfrak g \to \G$ is a global diffeomorphism from $\mathfrak g$ to $\G$.
Hence, if we choose a basis $\{X_1,\dots , X_N\}$ of $\mathfrak g$,  any $P\in \G$ can be written in a unique way as $P=\exp (p_1X_1+\dots +p_NX_N)$ and we can identify $P$ with the $N$-tuple $(p_1,\dots , p_N)\in \R^N$ and $\G$ with $(\R^N,\cdot, \delta_\lambda)$. The identity of $\G$ is the origin of $\R^N$.
 


For any $\lambda >0$, the (non isotropic) dilation $\delta _\lambda :\G\to \G$ are automorhisms of $\G$ and are  defined as 
\[
\delta _\lambda (p_1,\dots , p_N)=(\lambda ^{\alpha _1} p_1,\dots , \lambda ^{\alpha _N}p_N)
\]
where $\alpha _i \in \N$ is called homogeneity of the variable $p_i$ in $\G$ and is given by $\alpha _i=j$ whenever $m_{j-1}<i\leq m_j$ with  $m_s-m_{s-1}=n_s$ and $m_0=0$. Hence $1=\alpha _1=\dots =\alpha _{m_1}<\alpha _{m_1+1}=2\leq \dots \leq \alpha _N=\kappa$.

The explicit expression of the group operation $\cdot $ is determined by the Campbell-Hausdorff formula. It has the form 
\begin{equation*}
P\cdot Q= P+Q+\Q(P,Q) \quad \mbox{for all }\, P,Q \in \G\equiv \R^N,
\end{equation*} 
where $\Q=(\Q_1,\dots , \Q_N):\R^N\times \R^N \to \R^N$ and every $\Q_i$ is a homogeneous polynomial of degree $\alpha _i$ with respect to the intrinsic dilations of $\G$, i.e.
\begin{equation*}
\Q_i (\delta _\lambda P,\delta _\lambda Q)= \lambda ^{\alpha _i} \Q_i(P,Q) \quad \mbox{for all }\, P,Q \in \G \,  \mbox{ and }\, \lambda >0.
\end{equation*} 
We collect now further properties of $\Q$ following from Campbell-Hausdorff formula. First of all $\Q$ is antisymmetric, that is 
\begin{equation*}
\Q_i(P,Q)=-\Q_i(-P,-Q)  \quad \mbox{for all }\, P,Q \in \G .
\end{equation*}   
Moreover 
\[
\Q_i(P,0)=\Q_i(0,P)=0 \quad \mbox{and} \quad \Q_i(P,P)=\Q_i(P,-P)=0, 
\] 
and each $\Q_i (P,Q)$ depend only on the first components of $P$ and $Q$. More  precisely if $m_{i-1} < s\leq m_i $ and $i\geq 2$,
\begin{equation}\label{precedenti}
\begin{aligned}
\Q_1(P,Q) & =\dots =\Q_{m_1}(P,Q)=0\\
\Q_s (P,Q)& =\Q_s((p_1,\dots ,p_{m_{i-1}}), (q_1,\dots , q_{m_{i-1}})).
\end{aligned} 
\end{equation}

Observe also that $\G=\G^1\,  \oplus \, \G^2 \, \oplus \dots \oplus \, \G^\kappa $ where $\G^i=\exp (V_i)=\R^{n_i}$ is the $i^{th}$ layer of $\G$ and to write $P\in \G$ as $(P^1,\dots , P^\kappa)$ with $P^i\in \G^i$. According to this
\begin{equation}\label{opgr}
P\cdot Q= (P^1+Q^1,P^2+Q^2+\Q^2(P^1,Q^1),\dots ,P^\kappa +Q^\kappa+\Q^\kappa ((P^1,\dots , P^{\kappa-1} ),(Q^1,\dots ,Q^{\kappa-1}) ) 
\end{equation} 
for every $P=(P^1,\dots , P^\kappa )$, $Q=(Q^1,\dots ,Q^\kappa ) \in \G$. 
In particular $P^{-1}=(-P^1,\dots , -P^\kappa )$.

The norm of $\R^{n_s}$  is denoted with  the symbol $|\cdot |_{\R^{n_s}}$.

For any $P\in \G$ the intrinsic left translation $\tau _P:\G \to \G $ are  defined as
\begin{equation*}
Q \mapsto \tau _P Q := P\cdot Q=PQ.
\end{equation*}

A \emph{homogeneous norm} on $\G$ is a nonnegative function $P\mapsto \|P\|$ such that for all $P,Q\in \G$ and for all $\lambda \geq 0$
\begin{equation*}\label{defihomogeneous norm}
\begin{split}
\|P\|=0\quad &\text{if and only if }  P=0\\
\|\delta _\lambda P\|= \lambda \|P\|,& \qquad 
\|P \cdot Q\|\leq \|P\|+ \|Q\|.
\end{split}
\end{equation*}

There is a particular homogeneous norm defined in Theorem 5.1 in \cite{biblio8}, as following
\begin{equation}\label{epsilkappa}
\|(P^1,\dots , P^\kappa )\|:=\max_{s=1,\dots , \kappa } \big\{  \epsilon _s |P^s| ^{1/s}_{\R^{n_s}} \big\} \qquad \text{for all $(P^1,\dots , P^\kappa )\in \G$}
\end{equation} 
 with $\epsilon _1=1$, and $\epsilon _s \in(0,1]$ depending on the structure of the group for $s=2,\dots, \kappa$. This homogeneous norm is symmetric, i.e. $\|P\|=\|P^{-1}\|$ for all $P\in \G$ and such that $$ \|(P^1,0\dots , 0 )\|=|P^1|_{\R^{m_1}},$$ for all $P^1\in \R^{m_1}.$ 
 
 W.l.o.g. we consider the homogeneous norm defined in \eqref{epsilkappa}. Indeed from the fact that all the homogeneous norms are equivalent, all the estimates we will use hold with an arbitrary homogeneous norm. 

Given any homogeneous norm $\|\cdot \|$, it is possible to introduce a distance in $\G$ given by
\[
d(P,Q)=d(P^{-1} Q,0)= \|P^{-1} Q\| \qquad \text{for all $P,Q\in \G$}.   
\]
We observe that any  distance $d $ obtained in this way is always equivalent with the Carnot-Carath\'eodory's distance $d_{cc}$ of the group (see Proposition 5.1.4 and Theorem 5.2.8 \cite{biblio3}). 



The distance $d$ is well behaved with respect to left translations and dilations, i.e. for all $P,Q,Q' \in \G$ and $\lambda >0$,
\begin{equation*}
\begin{aligned}
d (P\cdot Q,P\cdot Q')=d(Q,Q'), \qquad d (\delta_\lambda Q,\delta_\lambda Q')=\lambda d(Q,Q')
\end{aligned}
\end{equation*}
Moreover, by Proposition 5.15.1 in \cite{biblio3}, for any bounded subset $\Omega \subset \G$ there exist positive constants $c_1=c_1(\Omega),c_2=c_2(\Omega)$ such that for all $P,Q \in \Omega$
\[
c_1|P-Q|_{\R^N} \leq d(P,Q) \leq c_2 |P-Q|_{\R^N}^{1/\kappa }
\] 
and, in particular, the topology induced on $\G$ by $d$
 is the Euclidean topology.

We also define the distance $dist_d$ between two set $\Omega_1, \Omega_2 \subset \G$ by putting
\[
dist_d(\Omega_1,\Omega_2) := \max \{ \sup_{Q' \in \Omega_2} d (\Omega_1, Q'),\, \sup_{Q\in \Omega_1} d (Q, \Omega_2) \},
\] 
where $d (\Omega_1, Q'):= \inf \{ d(Q,Q')\, :\, Q \in \Omega_1\}$.


The Hausdorff dimension of $(\G, d )$ as a metric space is  denoted \textit{homogeneous dimension} of $\G$ and it can be proved to be   the integer $\sum_{j=1}^{N}\alpha _j=\sum_{i=1}^{\kappa }i$ dim$V_i \geq N$ (see \cite{biblioMitchell}).

The subbundle of the tangent bundle $T\G$, spanned by the vector fields $X_1,\dots, X_{m_1}$ plays a particularly important role in the theory, and is called the \textit{horizontal bundle} $H\G$; the fibers of $H\G$ are
\[
H\G_P=\mbox{span} \{ X_1(P),\dots ,X_{m_1}(P)\}, \hspace{0,5 cm} P\in \G.  
\]
A sub Riemannian structure is defined on $\G$, endowing each fiber of $H\G$ with a scalar product $\langle\cdot ,\cdot \rangle_P$ and a norm $|\cdot |_P$ making the basis $X_1(P),\dots ,X_{m_1}(P)$ an orthonormal basis. Hence, if $v=\sum_{i=1}^{m_1}v_i X_i(P)=v^1$ and $w=\sum_{i=1}^{m_1}w_i X_i(P)=w^1$ are in  $H\G$, then $\langle v,w \rangle_P:=\sum_{i=1}^{m_1} v_i w_i$ and $|v|_P^2:=\langle v,v \rangle_P$. We will write, with abuse of notation, $\langle\cdot ,\cdot \rangle$  meaning $\langle\cdot ,\cdot \rangle_P$ and $|\cdot |$ meaning $|\cdot |_P$.

The sections of $H\G$ are called \textit{horizontal sections}, a vector of $H\G_P$ is an \textit{horizontal vector} while any vector in $T\G _P$ that is not horizontal is a vertical vector. 

The Haar measure  of the group $\G=\R^N$ is the Lebesgue measure $d\mathcal{L}^N$. It is left (and right) invariant.
Various Lebesgue spaces on $\G$ are meant always with respect to the measure  $d\mathcal{L}^N$ and are denoted as $\mathcal L^p(\G)$.

\subsection{$\C^1_\G$ functions, $\G$-regular surfaces, Caccioppoli sets} (See \cite{biblio9, biblioLIBROSC}). 
In \cite {biblio11} Pansu introduced an appropriate notion of differentiability for functions acting between Carnot groups. We recall this definition in the particular instance that is relevant  here. 

Let $\mcal U$ be an open subset of a Carnot group $\G$.  A function $f:\mathcal U\to \R^k$ is Pansu differentiable or more simply P-differentiable in $A_0\in \mathcal U$ if there is a homogeneous homomorphism 
\[
d_\mathbf Pf(A_0): \G\to \R^k,
\]
the Pansu differential of $f$ in $A_0$, such that, for $B\in \mathcal U$, 
\[
\lim_{r\to 0^+}\sup_{0<\Vert A_0^{-1}B\Vert<r}\frac{|f(B)-f(A_0)- d_\mathbf Pf(A_0)(A_0^{-1}B)|_{\R^k}}{\Vert A_0^{-1}B\Vert}= 0.
\]
Saying that $d_\mathbf Pf(A_0)$ is a  homogeneous homomorphism we mean that $d_\mathbf Pf(A_0): \G\to \R^k$ is a group homomorphism and also that $d_\mathbf Pf(A_0)(\delta_\lambda B)=\lambda d_\mathbf Pf(A_0)(B)$ for all $B\in \G$  and $\lambda \geq 0$.

Observe that, later on in Definition \ref{d3.2.1}, we give a different notion of differentiability for functions acting between subgroups of a Carnot group and we reserve the notation $df$ or $df(A_0)$ for that differential. 

We denote  $\C^1_\G (\mcal U ,\R^k )$  the set of functions $f:\mathcal U\to \R^k$ that are P-differentiable in each $A\in \mathcal U$ and such that $d_\mathbf Pf(A)$ depends continuously on $A$. 

It can be proved that $f=(f_1,\dots , f_k)\in \C^1_\G (\mcal U ,\R^k )$ if and only if the distributional horizontal derivatives  
$X_if_j $,  for $i=1\dots, m_1$, $j=1,\dots, k$,
are continuous in $\mcal U $.  Remember that $\C^1(\mcal U, \R ) \subset \C^1_\G (\mcal U, \R)$ with strict inclusion whenever $\G$ is not abelian (see Remark 6 in \cite{biblio6}).

The \emph{horizontal Jacobian} (or the \emph{horizontal gradient} if $k=1$) of $f:\mcal U \to \R^k$ in $A\in \mathcal U$ is the matrix
\[
 \nabla_\G f(A):=\left[X_if_j(A)\right] _{i=1\dots m_1, j=1\dots k}
\]
when the partial derivatives $X_if_j$ exist. Hence $f=(f_1,\dots , f_k)\in \mathbb C^1_\G (\mcal U ,\R^k )$ if and only if its horizontal Jacobian exists and is continuous in $\mathcal U$. 
 The \emph{horizontal divergence} of $\phi:=(\phi _1,\dots , \phi _{m_1}):\mcal U\to \R^{m_1}$  is defined as 
\begin{equation*}
 \mbox{div}_\G \phi := \sum_{j=1}^{m_1} X_j\phi _j
\end{equation*}
if $X_j\phi _j$ exist for  $j=1,\dots ,m_1$.

The following proposition shows that the P-differential of  a P-differentiable map $f$ is represented by horizontal gradient $\nabla_\G f$:
\begin{prop}\label{propP-differential2}
If $f:\mathcal{U} \subset \G \to \R$ is P-differentiable at a point $P$ and $d_\mathbf Pf(P)$ is P-differential of $f$ at $P$, then
\begin{equation*}
d_\mathbf Pf(P)(Q)=
 \nabla_\G f(P) Q^1, \quad \text{for all } \, Q=(Q^1,\dots ,Q^\kappa ) \in \mathcal{U}.
\end{equation*}
\end{prop}

In the setting of Carnot groups, there is a natural definition of bounded variation functions and of finite
perimeter sets (see \cite{biblioGARN} or \cite{biblioLIBROSC} and the bibliography therein).

We say that $f:\mcal U \to \R $ is of bounded $\G$-variation in an open set $\mcal U \subset \G$ and we write $f\in BV_\G(\mcal U )$, if $f\in \mathcal L^1(\mcal U )$ and
\[
\| \nabla _\G f \| (\mcal U ):= \sup \Bigl\{ \int _{\mcal U} f \, \mbox{div}_\G\phi \, d\mathcal{L}^N : \phi \in \C^1_c (\mcal U , H\G), |\phi (P)| \leq 1 \mbox{ for all } P \in \mcal U  \Bigl\} <+\infty .
\] 
The space $BV_{\G , loc }(\mcal U )$ is defined in the usual way.
A set $\mcal E\subset \G$ has locally finite $\G$-perimeter, or is a $\G$-Caccioppoli set, if $\chi_{\mcal E} \in BV_{\G , loc }(\G)$, where $\chi_{\mcal E}$ is the characteristic function of the set $\mcal E$. In this case the measure $\| \nabla _\G \chi_{\mcal E}\|$ is called the $\G$-perimeter measure of $\mcal E$ and is denoted by $|\partial \mcal E|_\G$.

Now we define co-abelian intrinsic submanifold as in Definition 3.3.4 \cite{biblioKOZHEVNIKOV}. Following the terminology of \cite{biblio24}, we call these objects $k$-codimensional $\G$-regular surfaces.
\begin{defi}
$S\subset \G$ is a \emph{$k$-codimensional $\G$-regular surface} if for every $P\in S$ there are a neighbourhood $\mcal U$ of $P$ and a function $f=(f_1,\dots , f_k)\in \mathbb{C}^1_\G(\mcal U,\R^k)$ such that
\[
S\cap \mcal U=\{ Q\in \mcal U : f(Q)=0 \}
\]
and $d_\mathbf P f(Q)$ is surjective, or equivalently if the $(k\times m_1)$ matrix $\nabla_\G f(Q)$ has rank $k$, with $k<m_1$, for all $Q\in \mcal U$.
\end{defi}

The class of $\G$-regular surfaces is different from the class of Euclidean regular surfaces. In \cite{biblio19}, the authors give an example of $\HH^1$-regular surfaces, in $\mathbb{H}^1$ identified with $\R^3$, that are (Euclidean) fractal sets. Conversely, there are continuously differentiable 2-submanifolds in $\R^3$ that are not $\HH^1$-regular surfaces (see \cite{biblio6} Remark 6.2 and  \cite{biblio1} Corollary 5.11).

Recall that a homogeneous subgroup $\W$ of $\G$ is a Lie subgroup such that $\delta _\lambda A\in \W$ for every $A\in \W$ and for all $\lambda >0$, we can give the following result about $\G$-regular surface:
\begin{theorem}[see \cite{biblioKOZHEVNIKOV}, Theorem 3.3.5]\label{biblioKOZHEVNIKOV}
Let $S\subset \G$ be a closed connected set. The following conditions are equivalent:
\begin{enumerate}
\item S is a $k$-codimensional $\G$-regular surface
\item S is Reifenberg vanishing flat with respect to a family of closed homogeneous subsets $\{ \W_P\, :\, P\in S\}$ such that $\W_P$ is a vertical subgroup of codimension k for some $P\in S$, i.e. for every relatively compact subset $S'\Subset S$ there is an increasing function $\beta :(0,\infty ) \to (0,\infty )$, with $\beta (t)\to 0^+$ when $t\to 0^+$, such that 
\begin{equation*}
\mbox{dist$_d$} \left( \mcal U(P,r)\cap S, \mcal U(P,r)\cap (P\cdot \W_P) \right) \leq \beta (r)r, \quad r>0
\end{equation*}
for any $P\in S'$.
\end{enumerate}

In particular, if $(1)$ or equivalently $(2)$ hold, i.e. $S$ is locally level set of a certain $f\in \mathbb{C}^1_\G(\G,\R^k)$ then
$$ \W_P= ker\left(d_\mathbf Pf(P)\right), \qquad \forall \, P \in S. $$ 

\end{theorem}

%


\subsection{Complementary subgroups and graphs} 
\begin{defi} We say that $\W$ and $\M$ are \emph{complementary subgroups in $\G$} if 
$\W$ and $\M$ are homogeneous subgroups of $\G$ such that  $\W \cap \M= \{ 0 \}$ and  $$\G=\W\cdot \M.$$  By this we mean that for every $P\in \G$ there are $P_\W\in \W$ and $P_\M \in \M$ such that $P=P_\W  P_\M$.
\end{defi}


The elements $P_\W \in \W$ and $P_\M \in \M$ such that $P=P_\W \cdot P_\M$ are unique because of $\W \cap \M= \{ 0 \}$ and are denoted  components of $P$ along $\W$ and $\M$ or  projections of $P$ on $\W$ and $\M$.
The projection maps $\mathbf{P}_\W :\G \to \W$ and $\mathbf{P}_\M:\G \to \M$ defined
\[
\mathbf{P}_\W (P)=P_\W, \qquad \mathbf{P}_\M (P)=P_\M, \qquad \text{for all $P\in \G$}
\]
are polynomial functions (see Proposition 2.2.14 in \cite{biblio22}) if we identify $\G$ with $\R^N$, hence are $\C^\infty$. Nevertheless in general they are not  Lipschitz maps, when $\W$ and $\mathbb{M}$ are endowed with the restriction of the left invariant distance $d$ of $\G$ (see Example 2.2.15 in \cite{biblio22}). 

\begin{rem}\label{rem2.2.1}
The stratification of $\G$ induces a stratifications on the complementary subgroups $\W$ and $\M$. If $\G=\G^1\oplus \dots \oplus \G^\kappa$ then also $\W= \W^1\oplus \dots \oplus \W^\kappa$, $\M=\M^1\oplus \dots \oplus \M^\kappa$ and $\G^i =\W^i \oplus \M^i$. A subgroup is \emph{horizontal} if it is contained in the first layer $\G^1$. If $\M$ is horizontal then the complementary subgroup $\W$  is normal.
\end{rem}

\begin{prop}[see \cite{biblio2}, Proposition 3.2]
If $\W$ and $\M$ are complementary subgroups in $\G$ there is $c_0=c_0(\W , \M)\in (0,1)$ such that for each  $P_\W \in \W$ and $P_\M \in \M$
\begin{equation}\label{c_0}
c_0(\| P_\W \|+\|P_\M \|)\leq \| P_\W  P_\M \| \leq \| P_\W \|+\|P_\M \|
\end{equation}
\end{prop}

\begin{defi}
 We say that $S\subset \G$ is a \emph{left intrinsic graph} or more simply a \emph{intrinsic graph} if there are complementary subgroups $\W$ and $\M$  in $\G$ and  $\phi: \E \subset \W \to \M$ such that
\[
S=\graph {\phi} :=\{ A \phi (A):\, A\in \E \}.
\]
\end{defi}
Observe that, by uniqueness of the components along $\W$ and $\M$, if $S=\graph {\phi}$ then $\phi $ is uniquely determined among all functions from $\W$ to $\M$.  

\begin{prop}[see Proposition 2.2.18 in \cite{biblio22}]\label{P2.2.18} 
If $S$ is a intrinsic graph then, for all $\lambda >0$ and for all $Q\in \G$, $Q \cdot S$ and $\delta _\lambda S$ are intrinsic graphs. In particular, if $S=\graph {\phi}$ with $\phi :\E \subset \W \to \M$, then
\begin{enumerate}
\item
For all $\lambda >0$, \[\delta _\lambda \left(\graph {\phi}\right) =\graph {\phi _\lambda}\]  where
$\phi _\lambda :\delta _\lambda \E \subset \W \to \M $ and 
$ \phi _\lambda (A):= \delta _\lambda \phi (\delta _{1/\lambda }A)$,  for $A \in \delta _\lambda \E$.  
\item
For any $Q\in \G$, \[Q \cdot \graph {\phi} = \graph {\phi _Q }\] where
$\phi _Q : \E _Q \subset \W \to \M$ is defined as 
$\phi _Q (A):= (\mathbf P_\M (Q^{-1}A))^{-1} \phi( \mathbf P_\W (Q^{-1}A))$, for all $A \in \E_Q:=\{ A\, :\, \mathbf P_\W (Q^{-1}A)\in \E  \}$.  
\end{enumerate}
\end{prop}
The following notion of intrinsic Lipschitz function appeared for the first time in \cite{biblio6} and was studied, more diffusely, in \cite{biblio17, biblio27, biblio21, biblio22,  biblio24, biblioPROF}. Intrinsic Lipschitz functions play the same role as Lipschitz functions in Euclidean context. 
\begin{defi}
Let $\W ,\M$ be complementary subgroups in $\G$, $\phi : \E \subset \W \to \M$. We say that $\phi $ is \emph{intrinsic $C_L$-Lipschitz  in $\E$}, or simply intrinsic Lipschitz, if there is $C_L\geq 0$ such that
\[
\|\mathbf{P}_\M (Q^{-1} Q') \|  \leq C_L  \|\mathbf{P}_\W (Q^{-1} Q')\|,  \qquad \text{for all $Q,Q' \in \graph {\phi} $.}
\]
$\phi : \E \to \M$ is locally intrinsic Lipschitz in $\E$ if $\phi $ is intrinsic Lipschitz in $\E '$ for every $\E ' \Subset \E$. 
\end{defi}

\begin{rem} \label{lip0}
In this paper we are interested mainly in the special case  when $\M$ is a horizontal subgroup and consequently  $\W$ is a normal subgroup. 
Under these assumptions, for all $P= A\phi (A),Q=B \phi (B) \in \graph {\phi} $ we have 
\[\mathbf{P}_\M (P^{-1} Q)= \phi (A)^{-1}\phi(B),\quad  \mathbf{P}_\W (P^{-1} Q)=\phi (A)^{-1} A^{-1}B\phi (A).
\]
 Hence, if $\M$ is a horizontal subgroup, $\phi : \E \subset \W \to \M$ is intrinsic Lipschitz if 
\[
\|\phi (A)^{-1}\phi(B)\| \leq  C_L \|\phi (A)^{-1} A^{-1}B\phi (A) \| \qquad \text{for all  $A,B \in \E$.}
\]
Moreover,  if $\phi$ is intrinsic Lipschitz then  $\|\phi (A)^{-1} A^{-1}B\phi (A) \|$ is comparable with $\Vert{P^{-1} Q}\Vert$. Indeed from \eqref{c_0}
\begin{equation*}
\begin{split}
c_0\|\phi (A)^{-1} A^{-1}B\phi (A) \| &\leq \Vert{P^{-1} Q}\Vert\\ &\leq \|\phi (A)^{-1} A^{-1}B\phi (A) \|+\|\phi (A)^{-1}\phi(B)\|\\
&\leq (1+C_L) \|\phi (A)^{-1} A^{-1}B\phi (A) \|.
\end{split}
\end{equation*}
The quantity $\|\phi (A)^{-1} A^{-1}B\phi (A) \|$, or better a symmetrized version of it, can play the role of a $\phi$ dependent, quasi distance on $\E$.  See e.g. \cite{biblio1}. 
\end{rem}



In Euclidean spaces, i.e. when $\G$ is $\R^N$ and the group operation is the usual Euclidean sum of vectors, intrinsic Lipschitz functions are the same as Lipschitz functions. On the contrary, when $\G$ is a general non commutative Carnot group and $\W$ and $\M$ are complementary subgroups, the class of intrinsic Lipschitz functions from $\W$ to $\M$ is different from the class Lipschitz functions (see Example 2.3.9 in \cite{biblio21}). More precisely,  if $\phi :\W \to \M $ is intrinsic Lipschitz then in general does not exists a constant $C$ such that
\[
\|\phi (A)^{-1}\phi (B)\|\leq C \|A^{-1}B\| \quad \mbox{for } A,B\in \W
\]
not even locally. Nevertheless   the following weaker result holds true:
\begin{prop}[see Proposition 3.1.8 in \cite{biblio22}]\label{lip84} 
Let $\W ,\M$ be complementary subgroups in a step $\kappa$ Carnot group $\G$.  Let  $\phi :\E \subset \W \to \M$ be an intrinsic $C_L$-Lipschitz function. Then, for all $r>0$,
\begin{enumerate}
\item there is $C_1= C_1(\phi, r)>0$ such that 
\[
\|\phi (A)\| \leq C_1 \quad \text{ for all $A\in \E$ with $\|A\|\leq r$}
\]
\item there is $C_2= C_2(C_L, r)>0$ such that  $\phi$ is locally $1/\kappa $-H\"older continuous i.e. 
\begin{equation*}
\|\phi (A)^{-1}\phi (B)\|\leq C_2 \|A^{-1}B\|^{1/\kappa }\quad \text{for all $A, B$ with $\|A\| ,  \|B\| \leq r$.}
\end{equation*}
\end{enumerate}
\end{prop}

\section{Intrinsic differentiability}
The notion of $P$-differentiability makes sense and can be introduced also for functions acting between complementary subgroups of a Carnot group $\G$. Nevertheless $P$-differentiability  does not seem to be the right  notion in this context.  Indeed P-differentiability  is a property that can be lost after a function is shifted as in Proposition \ref{P2.2.18}.  Here we recall a different notion of differentiability, the so called \emph{intrinsic differentiability} that is, by its very definition, invariant under translations. 
A function is intrinsic differentiable  if it is locally well approximated by   \emph{intrinsic linear} functions that are functions whose graph is a homogeneous subgroup in $\G$.


\begin{defi}
Let $\W$ and $\M$ be complementary subgroups in $\G$. Then $\ell:\W\to \M$ is  \emph{intrinsic linear} if $\ell$ is defined on all of $\W$ and if $\graph {\ell} $ is a homogeneous subgroup of $\G$.
\end{defi}
Intrinsic linear functions can be algebraically caracterized as follows.
\begin{prop}[see Propositions 3.1.3 and 3.1.6 in \cite{biblio21}]\label{biblio21ddd}
Let $\W$ and $\M$ be complementary subgroups in $\G$. Then $\ell:\W \to \M$ is intrinsic linear if and only if
\begin{equation*}
\begin{split}
 \ell(\delta _\lambda A)&=\delta _\lambda (\ell(A)), \qquad \text{for all $A\in \W$ and $\lambda \geq 0$}\\
\ell(AB) &=(\mathbf{P}_\HH (\ell(A)^{-1} B))^{-1} \ell(\mathbf{P}_\W(\ell(A)^{-1}B)), \qquad \mbox{for all } A,B\in \W .
\end{split}
\end{equation*}
Moreover any intrinsic linear function 
$\ell$ is a polynomial function 
and it is intrinsic Lipschitz with Lipschitz constant $C_L:=\sup\{ \| \ell(A)\|\, : \, \|A\| =1 \}$.
Note that $C_L < +\infty $ because $\ell$ is continuous. Moreover 
\begin{equation*}
\|\ell(A)\| \leq C_L\| A\|, \qquad \mbox{for all } A\in \W. 
\end{equation*}
In particular, if $\W$ is normal in $\G$ then $\ell:\W \to \M$ is  intrinsic linear if and only if
\begin{equation}
\begin{split}\label{lineareB}
 \ell(\delta _\lambda A)&=\delta _\lambda (\ell(A)), \qquad \text{for all $A\in \W$ and $\lambda \geq 0$}\\
\ell(AB) &=\ell(A) \ell\big(\ell(A)^{-1} B\ell(A)  \big), \qquad \mbox{for all } A,B\in \W .
\end{split}
\end{equation}

\end{prop}

We use intrinsic linear functions to define intrinsic differentiability as in the usual definition of differentiability.
\begin{defi}\label{d3.2.1}
Let $\W$ and $\M$  be complementary subgroups in $\G$ and let $\phi :\mathcal O \subset \W \to \M$ with $\mathcal O$ open in $\W$. For $A\in \mathcal O$, let $P:=A\cdot \phi (A)$ and $\phi _{P^{-1}}: \mathcal O _{P^{-1}} \subset \W \to \M$ be the shifted function defined in Proposition $\ref{P2.2.18}$. We say that $\phi$ is \emph{intrinsic differentiable in $A$} if the shifted function $\phi_{P^{-1}}$ is intrinsic differentiable in $0$, i.e. if there is a intrinsic linear $d\phi_A:\W\to \M$ such that
 \begin{equation}\label{3.0}
\lim_{r\to 0^+}\sup_{0<\|B\|<r}\frac{\| d\phi_{A} (B)^{-1} \phi _{P^{-1}} (B) \|}{\|B\|} =0.
\end{equation}
The function $d\phi_A$ is the \emph{intrinsic differential of $\phi $ at $A$}.
\end{defi}


\begin{rem} Definition \ref{d3.2.1} is a natural one because of the following observations.

\emph{(i)} If $\phi$ is intrinsic differentiable in $A\in \mathcal O$, there is a unique  intrinsic linear function $d\phi_A$ satisfying $\eqref{3.0}$.  Moreover $\phi$ is continuous at $A$. (See Theorem 3.2.8 and Proposition 3.2.3 in \cite{biblio21}).

\emph{(ii)} The notion of intrinsic differentiability is invariant under group translations. Precisely, let $P:=A\phi (A), Q:=B\phi (B)$, then $\phi $ is intrinsic differentiable in $A$ if and only if $\phi _{QP^{-1}} := (\phi _{P^{-1}})_{Q}$ is intrinsic differentiable in $B$.

\emph{(iii)} The analytic definition of intrinsic differentiability of Definition $\ref{d3.2.1}$ has an equivalent geometric formulation. Indeed intrinsic differentiability in one point is equivalent to the existence of a tangent subgroup to the graph (see \cite{biblio21}, Theorem 3.2.8). Let $\phi :\W \to \M$ be such that $\phi(0)=0$. We say  that an homogeneous subgroup $\T$ of $\G$ is a tangent subgroup to $\graph \phi$ in $0$ if
\begin{enumerate}
\item $\T$ is a complementary subgroup of $\M$
\item in any compact subset of $\G$
\begin{equation*}
\lim_{\lambda \to \infty } \delta _\lambda \left(\graph \phi\right)  =\T
\end{equation*}
in the sense of Hausdorff convergence.
\end{enumerate}
Moreover in \cite{biblio21}, Theorem 3.2.8 the authors show that $\phi$ is intrinsic differentiable in $0$ if and only if $\graph \phi$ has a tangent subgroup $\T$ in $0$ and in this case $\T=\graph {d\phi_0}$.
\end{rem}

In addition to pointwise intrinsic differentiability, we are interested in an appropriate  notion of continuously intrinsic differentiable functions. For functions acting between complementary subgroups, one possible way is to introduce a stronger, i.e. uniform, notion of intrinsic differentiability in the general setting of Definition \ref{d3.2.1}. 

\begin{defi}\label{Ndef1.1}
Let $\W$ and $\M$  be complementary subgroups in $\G$ and $\phi :\mathcal O \subset \W \to \M$ with $\mathcal O$ open in $\W$. For any $A\in \mathcal O$, let $P:=A\cdot \phi (A)$ and $\phi _{P^{-1}}: \mathcal O _{P^{-1}} \subset \W \to \M$ be the shifted function defined in Proposition $\ref{P2.2.18}$. We say that $\phi$ is \emph{uniformly intrinsic differentiable in $A_0\in \mathcal O$} or $\phi$ is \emph{u.i.d. in $A_0$} if  there exist a intrinsic linear function $d\phi_{A_0}: \W \to \M$ such that
\begin{equation}\label{3.0.1}
\lim_{r\to 0^+}\sup_{\|A_0^{-1}A\|<r}\sup_{0<\|B\|<r}\frac{\| d\phi_{A_0} (B)^{-1} \phi _{P^{-1}} (B) \|}{\|B\|} =0.
\end{equation}
Analogously, $\phi$ is u.i.d. in $\mathcal O$ if it is u.i.d. in every point of $\mathcal O$. 
\end{defi} 

\begin{rem} We recall that in Definition 3.16 in \cite{biblio2} the authors give another notion of uniformly intrinsic differentiable map. It is possible compare these two notions; indeed, for example, Proposition \ref{prop3.6cont} (3) is in the definition of u.i.d. for \cite{biblio2}. Moreover, it is clear, taking $A=A_0$ in \eqref{3.0.1}, that if $\phi$ is uniformly intrinsic differentiable in $A_0$ then it is intrinsic differentiable in $A_0$ and $d\phi_{A_0}$ is the intrinsic differential of $\phi $ at $A_0$ (that is the first point of the definition for \cite{biblio2}). 
\end{rem}
\medskip


From now on we restrict our setting  studying the notions of intrinsic differentiability and of uniform intrinsic differentiability for  functions $\phi: \W\to \HH$ when $\HH$ is a horizontal subgroup.  When $\HH$ is horizontal, $\W$ is always a normal subgroup since, as observed in Remark $\ref{rem2.2.1}$, it contains the whole strata $\G^2, \dots , \G^\kappa$.
In this case,   the more explicit form of the shifted function $\phi_{P^{-1}} $ allows  a more explicit form of equations \eqref{3.0} and \eqref{3.0.1}. 

First we observe that, when the target space is horizontal, intrinsic linear functions are Euclidean linear functions from the first layer of $\W$ to $\HH$. The analogous of the following proposition  is Proposition 3.23, \cite{biblio2} in the Heisenberg groups.




\begin{prop}\label{p2.5}
Let $\W$ and $\HH$ be complementary subgroups in $\G$ with $\HH$ horizontal.  Then a intrinsic linear function $\ell:\W \to \HH$ depends only on the variables in the first layer $\W^1:= \W\cap \G^1$ of $\W$. 
That is
\begin{equation}\label{lineare0}
\ell(A)=\ell(A^1,0,\dots,0), \qquad \text{ for all  $A=(A^1,\dots, A^\kappa)\in \W$}.
\end{equation}
Moreover there is $C_L\geq 0$ such that,  for all $A\in \W$, 
\begin{equation}\label{linea2.0}
\|\ell(A)\|\leq C_L \|(A^1,0,\dots,0)\|
\end{equation}
 and $\ell_{\vert \W^1}: \W^1\to \HH$ is Euclidean linear.

Finally if $k<m_1$ is the dimension of $\HH$ and if, without loss of generality, we assume that 
\[
\HH=\{P=(p_1,\dots,p_N): p_{k+1}=\dots =p_N=0\},\quad \W=\{P=(p_1,\dots,p_N): p_{1}=\dots =p_k=0\}\
\]  then there is a $k\times (m_1-k)$ matrix $\mathcal L$  such that 
\begin{equation}\label{DISSUdifferential}
\begin{split}
\ell(A)=\left(\mathcal L (a_{k+1},\dots,a_{m_1})^T, 0,\dots,0\right)
\end{split}
\end{equation}
  for all $A=(a_1,\dots,a_N)\in \W$ .
\end{prop}


\begin{proof}
In order to prove \eqref{lineare0}
first we prove that for all  $P=(P^1, \dots , P^\kappa )\in \W$, where $P^i\in \R^{n_i}$,
\begin{equation}\label{lineareA1}
\ell (P^1, \dots , P^\kappa )= \ell(P^1 ,  \dots , P^{\kappa -1},0).
\end{equation}
Let $A_\kappa :=(0,\dots ,0,P^\kappa )\in \W$. From $\eqref{precedenti}$ we have $\ell(A_\kappa)^{-1} \cdot A_\kappa \cdot \ell(A_\kappa) =A_\kappa$ and by $\eqref{lineareB}$ 
\begin{equation*}
\ell(A_\kappa \cdot A_\kappa)= \ell( A_\kappa) \cdot \ell\big(\ell(A_\kappa)^{-1}\cdot A_\kappa \cdot \ell(A_\kappa)  \big)
 =\ell( A_\kappa) \cdot  \ell( A_\kappa  )= 2\ell(A_\kappa).
\end{equation*}
Because $A_\kappa \cdot A_\kappa =(0,\dots ,0,2P^\kappa )=\delta _{2 ^{1/\kappa }}A_\kappa$, then 
$
2 \ell(A_\kappa ) = \ell(A_\kappa  \cdot A_\kappa )= \ell(\delta _{2 ^{1/\kappa }}A_\kappa  ) = 2 ^{1/\kappa } \ell(A_\kappa );
$
hence
\begin{equation}\label{lineareC}
\ell(A_\kappa)=0. 
\end{equation}
Because  $(P^1,  \dots , P^{\kappa})= A_\kappa \cdot (P^1 ,  \dots , P^{\kappa -1},0)$, from  $\eqref{lineareB}$ and $\eqref{lineareC}$ we get
\begin{equation*}
\begin{aligned}
\ell (P^1,  \dots , P^{\kappa})&=\ell\Big(A_\kappa  \cdot (P^1 ,  \dots , P^{\kappa -1},0)\Big)=
\ell(A_\kappa ) \cdot \ell\Big(\ell(A_\kappa)^{-1}  \cdot (P^1 ,  \dots , P^{\kappa -1},0) \cdot \ell(A_\kappa)  \Big)\\
& = \ell(P^1 ,  \dots , P^{\kappa -1},0),
\end{aligned}
\end{equation*}
and $\eqref{lineareA1}$ is proved.

In the next step we prove that 
\begin{equation}\label{lineare1}
\ell(P^1, \dots , P^\kappa )=\ell(P^1,\dots , P^{\kappa -2},0,0) 
\end{equation}
Let $A_{\kappa -1}:=(0,\dots , 0 ,P^{\kappa -1},0) \in \W$. From $\eqref{opgr}$, there is $\hat P^\kappa $, depending on $\ell(A_{\kappa -1})$ and $A_{\kappa -1}$, such that
\begin{equation*}
\ell(A_{\kappa -1})^{-1} \cdot A_{\kappa -1} \cdot \ell(A_{\kappa -1}) = (0,\dots , 0,P^{\kappa -1}, \hat P^\kappa).
\end{equation*}
From $\eqref{lineareB}$, $\eqref{lineareC}$ and the fact that $ (0,\dots , 0,P^{\kappa -1}, \hat P^\kappa) = (0,\dots , 0, \hat P^\kappa) \cdot A_{\kappa -1}$ we get
 \begin{equation*}
\begin{aligned}
\ell(0,\dots , 0,P^{\kappa -1}, \hat P^\kappa)&=\ell((0,\dots , 0, \hat P^\kappa)\cdot  A_{\kappa -1})\\
&=\ell(0,\dots , 0, \hat P^\kappa)\cdot  \ell\Big(\ell(0,\dots , 0, \hat P^\kappa)^{-1}  \cdot A_{\kappa -1} \cdot \ell(0,\dots , 0, \hat P^\kappa)  \Big)\\
&=\ell(A_{\kappa -1})
\end{aligned}
\end{equation*}
and consequently 
\begin{equation*}
\ell(A_{\kappa -1}\cdot A_{\kappa -1})=\ell(A_{\kappa -1})\cdot  \ell \Big( \ell(A_{\kappa -1})^{-1}\cdot  A_{\kappa -1} \cdot \ell(A_{\kappa -1}) \Big) =\ell(A_{\kappa -1}) \cdot \ell(A_{\kappa -1}).
\end{equation*}
Because $A_{\kappa -1}\cdot A_{\kappa -1} =(0,\dots ,0,2P^{\kappa -1},0 )=\delta _{2 ^{1/(\kappa -1) }}A_{\kappa -1}$   we have
\begin{equation*}
2 \ell(A_{\kappa -1}) = \ell(A_{\kappa -1})\cdot \ell(A_{\kappa -1})= \ell(A_{\kappa -1}\cdot A_{\kappa -1})= \ell(\delta _{2 ^{1/{(\kappa -1)}}}A_{\kappa -1}) = 2 ^{1/{(\kappa -1)}} \ell(A_{\kappa -1}).
\end{equation*}
Then
\begin{equation}\label{equationlineare4}
 \ell(A_{\kappa -1})=\ell(0,\dots , 0 ,P^{\kappa -1}, \hat P^{\kappa })=0.
\end{equation}
Because  $(P^1, \dots , P^\kappa )=(P^1,\dots , P^{\kappa -2},0,0)\cdot (0,\dots , 0 ,P^{\kappa -1}, \bar P^{\kappa })$ for appropriate $\bar P^{\kappa }$ and\\ $\ell(0,\dots , 0 ,P^{\kappa -1}, \bar P^{\kappa })=0,$ \footnote{ Because $(0,\dots , 0 ,P^{\kappa -1}, \bar P^{\kappa })=(0,\dots , 0 , \tilde P^{\kappa }) (0,\dots , 0 ,P^{\kappa -1}, \hat P^{\kappa })$ for some $\tilde P^{\kappa }$, from  $\eqref{lineareB}$, $\eqref{lineareC}$ and \eqref{equationlineare4} we get
\begin{equation*}
\begin{aligned}
\ell (0,\dots , 0 ,P^{\kappa -1}, \bar P^{\kappa })&=\ell\Big((0,\dots , 0 , \tilde P^{\kappa }) (0,\dots , 0 ,P^{\kappa -1}, \hat P^{\kappa })\Big)\\
&=\ell(0,\dots , 0, \tilde P^\kappa)\cdot  \ell\Big(\ell(0,\dots , 0, \tilde P^\kappa)^{-1}  \cdot (0,\dots , 0 ,P^{\kappa -1}, \hat P^{\kappa }) \cdot \ell(0,\dots , 0, \tilde P^\kappa)  \Big)\\
& = \ell (0,\dots , 0 ,P^{\kappa -1}, \hat P^{\kappa }) =0.
\end{aligned}
\end{equation*}} we obtain $\eqref{lineare1}$ from $\eqref{lineareB}$.
This procedure can be iterated to get \eqref{lineare0}.

Now it is easy to see that $\ell$ is Euclidean linear. Indeed for all $A,B\in \W$ and $\lambda>0$
\[
 \ell((\lambda A^1,0\dots,0))=\ell((\delta_\lambda A)^1,0\dots,0)=\ell (\delta_\lambda A)=\delta_\lambda \ell(A)=\lambda \ell(A^1,0\dots,0)
\]
and
\begin{equation*}
\begin{split}
& \ell(( A^1,0\dots,0)+(B^1,0\dots,0))= \ell( (AB)^1,0\dots,0)=\ell(AB)\\
 &=\ell(A)\ell\left(\ell(A)^{-1}B\ell(A)\right)\\
 &=\ell( A^1,0\dots,0)\ell\left((\ell(A)^{-1}B\ell(A))^{1},0,\dots,0)\right)\\
  &=\ell(A^1,0\dots,0)\ell(B^{1},0,\dots,0)=\ell(A^1,0\dots,0)+\ell(B^{1},0,\dots,0).
\end{split}
\end{equation*}

%
\end{proof}
Keeping in mind this special form of intrinsic linear functions we obtain the following special form of intrinsic differentiability. The reader can see Proposition 3.25 (ii) and Proposition 3.26 (ii) in \cite{biblio2} for the Heisenberg groups.
\begin{prop}\label{prop1.1.1}
 Let $\HH$ and $\W$ be complementary subgroups of $\G$, $\mathcal O$ open in $\W$ and $\HH$ horizontal. Then $\phi :\mathcal O \subset \W\to \HH$ is intrinsic differentiable in $A_0\in \mathcal O$ if and only if there is a intrinsic linear  $d\phi_{A_0}:\W\to \HH$ such that
\[
\lim_{r\to 0^+}\sup_{0<\|A_0^{-1}B\|<r} \frac{\|     \phi (B) -\phi (A_0)- d\phi_{A_0}( A_0^{-1}B )\|} {\|  \phi (A_0)^{-1} A_0^{-1}B \phi (A_0)\|} =0.
\]
Analogously,  $\phi$ is \emph{uniformly intrinsic differentiable in $A_0\in \mathcal O$}  if  there is a intrinsic linear  $d\phi_{A_0}:\W\to \HH$ such that
\[
\lim_{r \to 0^+}\sup_{\|A_0^{-1}A\|<r}\sup_{0<\|A^{-1}B\|<r}     \frac { \|  \phi ( B) - \phi (A)  - d\phi_{A_0}(A^{-1} B) \|}{\|\phi(A)^{-1}A^{-1}B\phi(A)  \|}  =0
\]
where $r$ is small enough so that $\mcal U(A_0,2r)\subset \mcal O$.

\end{prop}
\begin{proof}
First notice that $\phi (B) -\phi (A)- d\phi_{A_0} ( A^{-1}B )=d\phi_{A_0} (A^{-1}B)^{-1} \phi (A)^{-1} \phi (B)$ because both $d\phi_{A_0}$ and $\phi$ are valued in the horizontal subgroup $\HH$. 
Because $\W$ is normal in $\G$, then for $P:=A\phi(A)$ and for all $B'\in \mathcal O_{P^{-1}}$ 
$$\phi _{P^{-1}}(B')= \phi (A)^{-1} \phi (A\phi (A) B' \phi (A)^{-1}).$$  
Then $\eqref{3.0}$ yields that  $\phi : \mathcal O \to \HH$ is intrinsic differentiable in $A_0\in \mathcal O$ if there is an intrinsic linear map $d\phi_{A_0} :\W\to \HH$ such that
 \begin{equation*}
\| d\phi_{A_0} (B')^{-1} \phi (A)^{-1} \phi (A\phi (A) B' \phi (A)^{-1}) \| =o(\|B'\|) \quad \mbox{ as } \|B'\|\to 0,
\end{equation*}
that, setting  $B:=A\phi (A) B' \phi (A)^{-1}$ is equivalent to 
 \begin{equation*}
\|d\phi_{A_0}\left(\phi(A)^{-1}A^{-1}B\phi(A)\right)^{-1} \phi (A)^{-1} \phi (B) \| =o(\|\phi(A)^{-1}A^{-1}B\phi(A)\|) \quad \mbox{ as } \|A^{-1}B\|\to 0.
\end{equation*}
Finally, from Proposition \ref{p2.5} we know that $d\phi_{A_0}$ depends only on the variables in the first layer of $\W$. The group operation on the first layer is commutative hence $d\phi_{A_0} \left(\phi(A)^{-1}A^{-1}B\phi(A)\right)= d\phi_{A_0}\left(A^{-1}B\right)$.


\end{proof}

\begin{rem}
If $k<m_1$ is the dimension of $\HH$, and if, without loss of generality, we assume that
\[
\HH=\{P: p_{k+1}=\dots =p_N=0\}\qquad \W=\{P: p_{1}=\dots =p_k=0\}\
\]  then,  from \eqref{DISSUdifferential}, there is a $k\times (m_1-k)$ matrix, here denoted as   $\nabla^\phi\phi(A_0)$, such that 
\begin{equation}\label{comefattoequation}
d\phi_{A_0} (B)= \left(\nabla^\phi\phi(A_0) (b_{k+1},\dots,b_{m_1})^T,0,\dots ,0\right),
\end{equation}
for all $B=(b_1,\dots,b_N)\in \W$. The matrix $\nabla^\phi\phi(A_0)$ is called the \emph{intrinsic horizontal Jacobian} of $\phi$ in $A_0$ or the \emph{intrinsic horizontal gradient} or even the \emph{intrinsic gradient} if $k=1$.
\end{rem}

The point (1) of the following proposition states precisely a natural relation between uniform intrinsic differentiability and intrinsic Lipschitz continuity (see Proposition 3.30  \cite{biblio2} for the Heisenberg groups). The point (2) is a generalization of what was previously known for u.i.d. functions in Heisenberg groups (see  \cite{biblio1}, Theorem 1.3).

\begin{prop}\label{prop3.6cont}
Let $\HH$, $\W$ be complementary subgroups of $\G$ with $\HH$ horizontal. Let  $\mathcal O$ be open in $\W$ and $\phi :\mathcal O\to \HH$ be u.i.d. in $\mathcal O$. Then
\begin{enumerate}
\item $\phi$ is intrinsic Lipschitz continuous in every relatively compact subset of $\mathcal O$. 
\item $\phi \in h^{1/\kappa}_{loc}(\mathcal O)$, that is
$\phi\in \C(\mathcal{O},\R )$ and for all $\mathcal F\Subset \mathcal O$ and $A, B\in \mathcal{F}$
\begin{equation}\label{big3.3.11}
\lim_{r\to 0^+} \sup_{0<\|A^{-1} B\| <r}  \, \frac{\|\phi (B)-\phi (A)\|}{\|A^{ -1}B \|^{1/\kappa}}  =0.
\end{equation}

\item the function $A \mapsto d\phi_{A}$ is continuous in $\mathcal O$.
\end{enumerate}

\end{prop}

\begin{proof} 
(Proof of 1)  
For each $A_0\in \mathcal O$ there is $r=r(A_0)>0$ s.t. for all $A, B\in \mathcal{U}(A_0,r)\cap \mathcal O$ 
\[
\|  \phi ( B) - \phi (A)  - d\phi_{A_0}(A^{-1} B) \|\leq \|\phi(A)^{-1}A^{-1}B\phi(A)  \|.
\]
Moreover from Proposition \ref{p2.5}, because the intrinsic linear function $d\phi_{A_0}$ depends only on the variables on the first layer of $\W$, we have
\[
\Vert{d\phi_{A_0}(A^{-1} B)}\Vert\leq C_L | (A^{-1} B)^1|_{\R^{m_1}} \leq C_L \Vert \phi(A)^{-1}A^{-1}B\phi(A) \Vert,
\]
where $C_L$ is the intrinsic Lipschitz constant of $d\phi_{A_0}$. Finally
\begin{equation*}
\begin{split}
\|  \phi ( B) - \phi (A)\|&\leq \|  \phi ( B) - \phi (A)  - d\phi_{A_0}(A^{-1} B) \|+\Vert{d\phi_{A_0}(A^{-1} B)}\Vert\\
&\leq(1+C_L)  \|\phi(A)^{-1}A^{-1}B\phi(A)  \|.
\end{split}
\end{equation*}
Then (1) follows by a standard covering argument.

\medskip
(Proof of 2)    For $A_0\in \mathcal O$, $A,B \in \mathcal U(A_0,r)\cap \mathcal O$  and $r>0$
let
\[
\rho( r):= \sup_{0<\|A^{-1} B\| <r}  \frac{ \|  \phi (B)- \phi (A)-d \phi _{ A_0}(A^{-1} B)\|}{\|\phi(A)^{-1}A^{-1}B\phi(A)  \| } 
\]
then $\lim_{ r \to 0}\rho( r)=0$ because $\phi $ is u.i.d. at $A_0$. Moreover, by (1) of this Proposition, we know that $\phi $ is intrinsic Lipschitz in $ \mathcal U(A_0,r)\cap \mathcal O$ and by Proposition \ref{lip84} (1) we have that $\| \phi (A)\|<C_1$ for all $A\in \mathcal U(A_0,r)\cap \mathcal O$. 

We recall (see \cite{biblio22}, Lemma 2.2.10) that in any Carnot group $\G$ of step $\kappa$   there is $ C= C(\G)>0$ such that
\begin{equation}\label{17FS}
\| Q^{-1} PQ\| \leq \|P\| +  C\left( \|P\|^{\frac{1}{\kappa}} \|Q\|^{\frac{\kappa -1}{\kappa}}+  \|P\|^{\frac{\kappa -1}{\kappa}} \|Q\|^{\frac{1}{\kappa}} \right) \quad \mbox{for all } P, Q \in \G.
\end{equation}
From \eqref{17FS} with $P=A^{-1}B$ and $Q=\phi (A)$ we deduce the existence of $C=C(C_1)>0$ such that
\[
\|\phi(A)^{-1}A^{-1}B \phi(A)\|  \leq C \|A^{-1}B\|^{1/\kappa } \qquad \mbox{for all }  A, B \in  \mathcal U(A_0,r)\cap \mathcal O.
\]
Therefore using  \eqref{linea2.0}
 \begin{equation*}
\begin{aligned}
& \frac{\| \phi (B)-\phi (A) \|}{ \|A^{-1}B \|^{1/\kappa } }\\
  & \leq  \frac{\|\phi (B)-\phi (A)-d\phi_{A_0} (A^{-1} B) \|}{\|\phi(A)^{-1}A^{-1}B \phi(A)\| } \,   \frac{ \|\phi(A)^{-1}A^{-1}B \phi(A)\|}{ \|A^{-1}B\|^{1/\kappa } }  + \frac{\| d\phi_{A_0} (A^{-1} B)\|}{ \|A^{-1}B\|^{1/\kappa } } \\
            & \leq C \rho(r) + C_2 r^{1-1/\kappa }
\end{aligned}
\end{equation*}
for all $A, B \in  \mathcal U(A_0,r)\cap \mathcal O$ with $A\ne B$. Hence $ \frac{\| \phi (B)-\phi (A) \|}{ \|A^{-1}B \|^{1/\kappa } } \to 0$ for $r\to 0$ and the proof of (2) is complete.
\medskip

(Proof of 3) In order to prove that $A\mapsto d\phi_A$ is continuous in $A_0\in \mathcal O$ we prove that for all $\eps >0$ there is $r=r(\eps, A_0)>0$ such that 
\[
\Vert d\phi_{A_1}(P)- d\phi_{A_0}(P)\Vert <\eps \Vert P\Vert 
\]
for $A_1\in \mathcal U(A_0,r)\cap \W$ and for all $P\in \W$. 

Indeed, because $\phi $ is u.i.d. in $A_0$, for all $\eps>0$ there is $r_0=r_0(\eps, A_0)>0$ such that 
\[
\Vert \phi(B)-\phi(A)-d\phi_{A_0}(A^{-1}B)\Vert <\eps \Vert \phi(A)^{-1} A^{-1}B \phi(A)\Vert \qquad \text{for all $A,B\in \mathcal U(A_0,r_0)\cap \W$}.
\]
Let $A_1\in\mathcal U(A_0,r_0)\cap \W$.  By assumption $\phi$ is u.i.d. also in $A_1$ hence there is $r_1=r_1(\eps, A_1)>0$ such that $\mathcal U(A_1,r_1)\subset \mathcal U(A_0,r_0)$ and
\[
\Vert \phi(B)-\phi(A)-d\phi_{A_1}(A^{-1}B)\Vert <\eps \Vert \phi(A)^{-1} A^{-1}B \phi(A)\Vert \qquad \text{for all $A,B\in \mathcal U(A_1,r_1)\cap \W$}.
\] 
Hence 
\[
\Vert d\phi_{A_1}(A^{-1}B)- d\phi_{A_0}(A^{-1}B)\Vert <2\eps \Vert \phi(A)^{-1} A^{-1}B \phi(A)\Vert \qquad \text{for all $A,B\in \mathcal U(A_1,r_1)\cap \W$}.
\]
Let $\delta\in(0,1)$ be a constant to be chosen at the end.  For all $P\in \mathcal U(0,\delta r_1)\cap \W$ let $B:= A_1P$. Then $B\in  \mathcal U(A_1, r_1)$ and 
\[
\Vert d\phi_{A_1}(P)- d\phi_{A_0}(P)\Vert <2\eps \Vert \phi(A_1)^{-1} P \phi(A_1)\Vert.
\]
The proof is completed if we can put on the right hand side $\Vert  P\Vert$ instead of $\Vert \phi(A_1)^{-1} P \phi(A_1)\Vert$.

Observe that $d\phi_{A_1}$ and $d\phi_{A_0}$ depend only on the components in the first layer of $\W$. Consequently  we can change  the components $P^2, \dots , P^\kappa$ of $P$ without changing $d\phi_{A_1}(P)- d\phi_{A_0}(P)$.  

An explicit computation shows that
\[
\phi(A_1)^{-1} P \phi(A_1)=\left(P^1, P^2+\mathcal P^2(P^1, \phi(A_1)), \dots, P^\kappa+\mathcal P^\kappa(P^1,\dots,P^{\kappa -1}, \phi(A_1))\right)
\]
where $\mathcal P^2, \dots, \mathcal P^{\kappa}$ are polynomials. Moreover each polynomial $\mathcal P^i$ depends only on the variables in the layers from 1 to $i-1$. 

Now we can conclude. Given $P=(P^1,\dots, P^\kappa)\in \mathcal U(0,\delta r_1)\cap \W $, we define $\tilde P= (\tilde P^1,\dots, \tilde P^\kappa)$ putting iteratively
\begin{equation*}
\begin{split}
& \tilde P^1:= P^1\\
& \tilde P^2:= -\mathcal P^2(\tilde P^1, \phi(A_1))\\
&\dots\\
& \tilde P^\kappa:= -\mathcal P^\kappa(\tilde P^1,\dots,\tilde P^{\kappa -1}, \phi(A_1)).
\end{split}
\end{equation*}
Observe that, if $\delta$ is sufficiently small,  $\tilde P \in \mathcal U(0, r_1)$, moreover 
$\phi(A_1)^{-1} \tilde P \phi(A_1)=(\tilde P^1, 0,\dots,0)$   and $\Vert \phi(A_1)^{-1} \tilde P \phi(A_1)\Vert = \vert \tilde P^1 \vert =\vert P^1 \vert $. Finally 
\[
\Vert d\phi_{A_1}(P)- d\phi_{A_0}(P)\Vert=\Vert d\phi_{A_1}(\tilde P)- d\phi_{A_0}(\tilde P)\Vert\leq 2\eps \vert P^1|_{\R^{m_1}}
\]
for all $P\in \mathcal U(0,\delta r_1)\cap \W$. By Proposition \ref{biblio21ddd}, it holds for all $P\in \W$ and the proof is completed.

\end{proof}

\bigskip


\section{$\G$-regular surfaces}
The main result of this section and of the first part of this paper is Theorem \ref{teo4.1}. In it we prove that, if $\HH$ is a horizontal subgroup, the intrinsic graph of $\phi:\mathcal O \subset \W \to \HH$ is a $\G$-regular $k$-codimensional surface  if and only if $\phi$ is uniformly intrinsic differentiable in $\mathcal O$.

\begin{theorem}\label{teo4.1}
Let $\W$ and $\HH$ be complementary subgroups of a Carnot group $\G$ with $\HH$  horizontal and $k$ dimensional. 
Let $\mathcal O$ be open in $\W$, $\phi :\mathcal O \subset \W \to \HH $ and $S:= \graph{\phi}$. Then for every $A_0\in \mathcal O$ the following are equivalent:
\begin{enumerate}
\item there are a neighbourhood $ \mathcal U$ of $A_0\cdot \phi (A_0)$ and $f\in \C_\G^1(  \mathcal U; \R^k)$ such that 
\begin{equation*}
\begin{split}
& S \cap  \mathcal U=\{P\in  \mathcal  U: f(P)=0\}\\
& d_{\bf P}f(Q)_{\vert \HH}:\HH\to \R^k\quad \text{is bijective for all $Q\in \mathcal U$}
\end{split}
\end{equation*}
and $Q\mapsto \left(d_{\bf P}f(Q)_{\vert \HH}\right)^{-1}$ is continuous.
\item $\phi $ is u.i.d. in a neighbourhood $\mathcal O ' \subset \mathcal O$ of $A_0$. 
\end{enumerate}
Moreover, if {\rm(1)} or equivalently {\rm (2)}, hold then, for all $A\in \mathcal O$ the intrinsic differential $d\phi_A$ is
\[
d\phi_A=- \left(d_{\bf P}f(A\phi(A))_{\vert \HH} \right)^{-1}\circ d_{\bf P}f(A\phi(A))_{\vert \W}.
\]
\end{theorem}
\medskip

 \begin{rem}
If, without loss of generality, we choose a base $X_1,\dots, X_N$ of $\mfrak g$ such that $X_1,\dots, X_k$ are  horizontal vector fields, $\HH=\exp(\text{\rm span} \{X_1,\dots, X_k\})$ and $\W=\exp(\text{\rm span} \{X_{k+1},\dots, X_N\})$ then 
\[
\HH=\{P: p_{k+1}=\dots =p_N=0\}\qquad \W=\{P: p_{1}=\dots =p_k=0\},
\] 
$
\mbox{and, if } f=(f_1,\dots, f_k), \mbox{ then }  \nabla _\G f= \left(
\, \mathcal{M}_1 \, \, | \,\, \mathcal{M}_2 \,
\right)
$
where 
 \begin{equation*}
\mathcal{M}_1 := \begin{pmatrix}
X_1f_1 \dots  X_kf_1 \\
\vdots \qquad \ddots \qquad \vdots \\
X_1f_k \dots  X_kf_k
\end{pmatrix},\qquad\mathcal{M}_2 := \begin{pmatrix}
X_{k+1}f_1\dots  X_{m_1}f_1 \\
\vdots \qquad \ddots \qquad \vdots \\
X_{k+1}f_k \dots  X_{m_1}f_k
\end{pmatrix}.
\end{equation*}
Moreover, for all $Q\in \mathcal U$, for all $A\in \mathcal O$ and for all $P\in \G$ 
\[
\left(d_{\bf P}f(Q)\right)(P)= \left(\nabla _\G f\right(Q))P^1
\]
and the intrinsic differential is 
\begin{equation}\label{teo4.1.1}
\begin{split}
d\phi_A(B)&= \left(\left(\nabla^\phi \phi(A)\right)(b_{k+1},\dots,b_{m_1})^T,0,\dots, 0\right)\\
&=\left( \left( - \mathcal M_1(A\phi(A))^{-1}\mathcal M_2(A\phi(A))\right)(b_{k+1},\dots,b_{m_1})^T ,0,\dots, 0\right),
\end{split}
\end{equation}
for all $B=(b_1,\dots,b_N)\in \W$ $($see \eqref{comefattoequation}$)$.

\end{rem}

 The proof of Theorem \ref{teo4.1} requires both Whitney's Extension Theorem and Implicit Function Theorem in Carnot groups. The proof of Whitney's Extension Theorem   can be found in \cite{biblio8} for Carnot groups of step two only, but it is identical for general Carnot groups (see  \cite{biblioDDD}, Theorem 2.3.8). 
 
 \begin{theorem}[Implicit Function Theorem, see \cite{biblioMAGNANI}, Theorem 1.3] Let $ \mathcal U$ be an open subset of $\G$.  Let $f\in \C^1_\G( \mathcal U, \R^{k})$ and assume that  $\nabla _\G f (Q)$ has rank $k$ for all $Q\in  \mathcal U$. 
We assume that for a fixed $P\in  \mathcal U$ there are complementary subgroups  $\W$ and  $ \HH$ of $\G$ where $\W= \mbox{ker}(\nabla _\G f (P))$. Then there are $\mathcal I\subset \W$ and $\mathcal J \subset \HH$, open and such that  $P_\W\in\mathcal I$ and $P_\HH\in \mathcal J$ and a unique continuous function $\phi :\mathcal I \to \mathcal J$ such that
\[
\left\{Q\in \mathcal I \mathcal J:\; f(Q)=f(P) \right\}= \{ A \phi (A)\, :\, A\in \mathcal I \}
\]
where $\mathcal I\mathcal J= \{A B \, :\, A\in\mathcal  I , B\in \mathcal J\}$.
\end{theorem}

\begin{theorem}[Whitney's Extension Theorem]\label{white} 
Let $\mathcal{F} \subset \G$ be a closed set and let $f$ and $g$ be continuous functions where $f:\mathcal{F}\to \R^k$ and   $g:\mathcal{F}\to \mathbf M_{k\times m_1}$ the space of $k\times m_1$ matrices. 
For $\mathcal{K} \subset \mathcal{F}$, $P$ and $Q\in \mathcal K$, $\delta >0$ let
\[
\rho _\mathcal{K}(\delta ):= \sup_{0<\|Q^{-1} P \|<\delta} \frac{|f(P)-f(Q)- g(Q)  (Q^{-1} P)^1|_{\R^k}}{\|Q^{-1} P \|}
\]
where $g(Q)  (Q^{-1} P)^1$ is the usual product between matrix and vector. 
If,  for all compact set $\mathcal{K} \subset \mathcal{F}$,
\[
\lim_{\delta \to 0} \rho _\mathcal{K}(\delta )= 0
\]
then there exists $\hat{f}\in \mathbb{C}^1_\G(\G , \R^k)$ such that
\[
\hat{f}_{|\mathcal{F}} = f, \hspace{0,5 cm   } \nabla _\G \hat{f}_{|\mathcal{F}} = g.
\]
\end{theorem}
A proof is in \cite{biblio8} when $\G$ is a Carnot groups of step two.  The general case is in \cite{biblioDDD}.
\medskip

Before proving Theorem \ref{teo4.1} we state a Morrey type inequality for functions in $\C_\G^1(\G, \R^k)$ (see also Lemma 3.2.2 in \cite{biblio12}).
\begin{lem}\label{lem4.2}
Let $P \in \G$, $r_0>0$  and $f\in \C_\G^1 (\mathcal U(P,r_0),\R^k)$. Then there is $C=C(P,r_0)>0$ such that, for each $\bar Q\in \mathcal U(P,r_0/2)$ and $r\in (0, r_0/4)$, 
\begin{equation*}
\vert f(Q)-f(\bar Q)- \nabla _\G f(\bar Q ) (\bar Q^{-1} Q )^1 \vert_{\R^k} \leq Cr\|\nabla _\G f- \nabla_\G f(\bar Q)\|_{\mathcal{L}^\infty (\mathcal U(\bar Q,2r))}
\end{equation*}
for all  $Q \in \mathcal U(\bar Q,r)$.
\end{lem}
\begin{proof}
Let  $\hat{f}: \mathcal U(\bar Q,r) \to \R^k$ be defined as 
\[
\hat{f}(Q):=f(Q)- \nabla _\G f(\bar Q) (\bar Q^{-1} Q )^1  .
\]
Then (see  Theorem 1.1. in \cite{biblio13})  there are $p>1 $ and $\hat C>0$ such that for all $ Q \in \mathcal U(\bar Q,r)$
\[
\vert \hat{f}(Q)-\hat{f}(\bar Q)\vert_{\R^k} \leq \hat{C}r \left( \ave_{ \mathcal U(\bar Q,2r)} |\nabla _\G \hat{f}|^p \, d\mathcal L^N \right)^{1/p} 
\]
where $\ave_{ \mathcal U} \cdot \, d\mathcal L^N:= \frac{1}{\mathcal L^N(\mathcal U)}\int_{\mathcal U}\cdot \, d\mathcal L^N$. 
Then, from $\hat{f}(\bar Q)=f(\bar Q)$ and $\nabla _\G \hat{f} = \nabla _\G f - \nabla _\G f(\bar Q)$, we have
\begin{equation*}
\begin{split}
&\vert f( Q)-f(\bar Q)-  \nabla _\G f( \bar Q) (\bar Q^{-1}  Q )^1 \vert_{\R^k}  ={\vert \hat{f}(Q)-\hat{f}(\bar Q) \vert}_{\R^k} \\
&\qquad\qquad \leq 2\hat{C}\,r\, \left( \ave _{ \mathcal U(\bar Q,2r)} |\nabla _\G \hat{f}|^p \, d\mathcal L^N \right)^{1/p} \\
&\qquad\qquad= 2\hat{C}\,r\, \left( \ave _{ \mathcal U(\bar Q,2r)} | \nabla _\G f - \nabla _\G f(\bar Q) |^p \, d\mathcal L^N \right)^{1/p} \\
&\qquad\qquad\leq Cr\|\nabla _\G f- \nabla_\G f(\bar Q)\|_{\mathcal{L}^\infty ( \mathcal U(\bar Q,2r ))}.
\end{split}
\end{equation*}
\end{proof}

\begin{proof}[Proof of Theorem $\ref{teo4.1}$.]

$(1) \Rightarrow (2)$. 

Let $A_0 \in \mathcal O$,
 $r>0$ such that $\mathcal I(A_0, r):=\mathcal U (A_0,r)\cap \W \subset \mathcal O$ and define $\Phi:\mathcal O\to \G$ as $\Phi(A):=A\phi(A)$ for all $A\in \mathcal O$. 

The function $\phi$ is continuous in $\mathcal O$ as follows from a well known elementary argument in Implicit Function Theorem. Hence there is $\delta_r=\delta(A_0,r)>0$ such that  $\norm {\Phi(A)^{-1}\cdot \Phi(B)} \leq \delta_r$ for all  $A, B\in \mathcal I(A_0, r)$. Then, recalling that $d_{\mathbf P}f=\nabla_{\mathbb G} f$ acting on the first layer,  for all $A, B\in \mathcal I(A_0, r)$
\begin{equation}\label{4.7}
\begin{split}
\vert d_{\mathbf P}f(\Phi(A))&\big(\Phi (A)^{-1} \Phi (B)\big)\vert_{\R^k}\\
&= \vert f(\Phi (B))-f(\Phi (A))- d_{\mathbf P}f(\Phi(A))\big(\Phi (A)^{-1} \Phi (B)\big)  \vert_{\R^k}\\
   & \leq C \rho(\delta _r) \, \|\Phi (A)^{-1} \Phi (B) \| \\
    & \leq C \rho(\delta _r) \, \left(\|\mathbf P_\HH(\Phi (A)^{-1} \Phi (B))\|+  \|\mathbf P_\W(\Phi (A)^{-1} \Phi (B))\|  \right )
\end{split}
\end{equation}
where $C$ is the constant in Lemma $\ref{lem4.2}$, and 
\begin{equation*}
\rho(\delta_r):= \| \nabla_\G f -\nabla _\G f(\Phi(A_0))\|_{\mathcal{L}^\infty ( \mcal U(\Phi(A_0),2\delta_r ))} . 
\end{equation*}
Observe also that  
\[
\lim _{r \to 0} \rho(\delta _r)=0.
\]
The function $\phi$, beyond being simply continuous, is also intrinsic Lipschitz continuous i.e. (see Remark \ref{lip0})  there is a constant $C_L>0$ such that 
\begin{equation}\label{4.10}
\| \phi (A)^{-1}\cdot \phi (B)\| \leq C_L\|\phi(A)^{-1}A^{-1}B\phi(A)  \|.
\end{equation}
Indeed, because $\mathbf P_\HH(\Phi (A)^{-1} \Phi (B))= \phi (A)^{-1}\cdot \phi (B)$ and $\mathbf P_\W(\Phi (A)^{-1} \Phi (B))=\phi(A)^{-1}A^{-1}B\phi(A) $, from $\eqref{4.7}$, 
 we get for all $A,B\in \mathcal I(A_0, \delta_r)$
\begin{equation}\label{4.9}
\begin{split}
&\norm{\left(d_{\mathbf P}f(\Phi(A))_{\vert\HH}^{-1}\circ d_{\mathbf P}f(\Phi(A))_{\vert\W}\right)(A^{-1} B)\cdot \phi(A)^{-1}\cdot \phi(B)}\\
&= \norm{d_{\mathbf P}f(\Phi(A))_{\vert\HH}^{-1}\left( d_{\mathbf P}f(\Phi(A))_{\vert\W}(A^{-1} B)+ d_{\mathbf P}f(\Phi(A))_{\vert\HH}(\phi(A)^{-1}\cdot \phi(B))\right)}\\
&  \mbox{(because $d_{\mathbf P} f=\nabla_{\mathbb G} f$ which acts only on the first components, $d_{\mathbf P} f(A^{-1}B)=d_{\mathbf P}f(\phi(A)^{-1}A^{-1}B\phi(A))$}\\
 & \mbox{ and then we use the fact that $d_{\mathbf P} f$ is an homomorphism}) \\
&= \norm{d_{\mathbf P}f(\Phi(A))_{\vert\HH}^{-1}\left( d_{\mathbf P}f(\Phi(A))(\Phi(A)^{-1}\cdot \Phi(B))\right)}\\
&\leq C \rho(\delta _r) \, \norm{d_{\mathbf P}f(\Phi(A))_{\vert\HH}^{-1}} \left(\|\mathbf P_\HH(\Phi (A)^{-1} \Phi (B))\|+  \|\mathbf P_\W(\Phi (A)^{-1} \Phi (B))\|  \right )\\
&= C \rho(\delta _r) \, \norm{d_{\mathbf P}f(\Phi(A))_{\vert\HH}^{-1}} \left(\|\phi (A)^{-1}\cdot \phi (B)\|+  \|\phi(A)^{-1}A^{-1}B\phi(A))\|  \right ).
\end{split}
\end{equation}
Now we choose $r$ small enough such that for all $A\in \mathcal I(A_0, r)$,
\begin{equation*}
C \rho(\delta _r) \, \norm{d_{\mathbf P}f(\Phi(A))_{\vert\HH}^{-1}}   \leq \frac{1 }{2}.
\end{equation*}
Moreover, because $d_{\mathbf P}f(\Phi(A))_{\vert\W}=d_{\mathbf P}f(\Phi(A))_{\vert\W^1}$ and the homogeneous norm defined in \eqref{epsilkappa} restricted on the horizontal layer is the Euclidean norm, there is $C_2=C_2(A_0,r)>0$ such that for all $A, B\in \mathcal I(A_0, \delta_r)$,
\[
\norm{\left(d_{\mathbf P}f(\Phi(A))_{\vert\HH}^{-1}\circ d_{\mathbf P}f(\Phi(A))_{\vert\W}\right)(A^{-1} B)}\leq C_2\,\vert (A^{-1} B)^1\vert_{\R^{m_1}}.
\]
Putting all this together and by the homogeneous norm defined in \eqref{epsilkappa} is symmetric, 
for all $A,B\in \mathcal I(A_0, \delta_r)$ 
\begin{equation*}
\begin{aligned}
\norm{\phi (A)^{-1}\cdot \phi (B)}& \leq \norm{\left(d_{\mathbf P}f(\Phi(A))_{\vert\HH}^{-1}\circ d_{\mathbf P}f(\Phi(A))_{\vert\W}\right)(A^{-1} B)\cdot \phi(A)^{-1}\cdot \phi(B)}\\
& \qquad + \norm{\left(d_{\mathbf P}f(\Phi(A))_{\vert\HH}^{-1}\circ d_{\mathbf P}f(\Phi(A))_{\vert\W}\right)(A^{-1} B)}\\
     &\leq 1/2 \left(\|\phi (A)^{-1}\cdot \phi (B)\|+  \|\phi(A)^{-1}A^{-1}B\phi(A))\|  \right ) + C_2 |(A^{-1}  B)^1| _{\R^{m_1}}\\
     &\leq 1/2\, \|\phi (A)^{-1}\cdot \phi (B)\|+ (1/2 +C_2) \|\phi(A)^{-1}A^{-1}B\phi(A))\|  .
\end{aligned}
\end{equation*}
This inequality implies that $\eqref{4.10}$ holds. As a consequence from \eqref{4.9} there is $C_3=C_3(A_0,r)>0$ such that for all $A,B\in \mathcal I(A_0,\delta_ r)$
\begin{equation*}
\norm{\left(d_{\mathbf P}f(\Phi(A))_{\vert\HH}^{-1}\circ d_{\mathbf P}f(\Phi(A))_{\vert\W}\right)(A^{-1} B)\cdot \phi(A)^{-1}\cdot \phi(B)}\leq C_3 \rho(\delta_r) \norm{\phi(A)^{-1}A^{-1}B\phi(A)}.
\end{equation*}
Now the proof that $\phi $ is u.i.d. at $A_0$ follows readily. Indeed, for all $A,B\in \mathcal I(A_0, \delta _r)$,
\begin{equation*}
\begin{split}
&\norm{\left(d_{\mathbf P}f(\Phi(A_0))_{\vert\HH}^{-1}\circ d_{\mathbf P}f(\Phi(A_0))_{\vert\W}\right)(A^{-1} B)\cdot \phi(A)^{-1}\cdot \phi(B)}\\
& \mbox{(using again the fact } \|P\| = \|P^{-1}\| )\\
&\leq \norm{\left(d_{\mathbf P}f(\Phi(A))_{\vert\HH}^{-1}\circ d_{\mathbf P}f(\Phi(A))_{\vert\W}\right)(A^{-1} B)\cdot \phi(A)^{-1}\cdot \phi(B)}\\
&\qquad + \norm {\left(d_{\mathbf P}f(\Phi(A_0))_{\vert\HH}^{-1}\circ d_{\mathbf P}f(\Phi(A_0))_{\vert\W}\right)(A^{-1} B)\cdot\left( \left(d_{\mathbf P}f(\Phi(A))_{\vert\HH}^{-1}\circ d_{\mathbf P}f(\Phi(A))_{\vert\W}\right)(A^{-1} B)\right)^{-1}}\\
&\leq \left(C_3 \rho(\delta _r)+o(1)\right){{ \|\phi(A)^{-1}A^{-1}B\phi(A)  \| }}
\end{split}
\end{equation*}
That is $\phi $ is uniformly intrinsic differentiable at $A_0$ and moreover 
\[
d\phi_{A_0}=-d_{\mathbf P}f(\Phi(A_0))_{\vert\HH}^{-1}\circ d_{\mathbf P}f(\Phi(A_0))_{\vert\W}.
\]
This completes the proof of $(1) \Rightarrow (2)$.
\medskip

$(2) \Rightarrow (1)$  The proof of this second part uses Whitney's Extension Theorem to prove the existence of an appropriate function $f\in \mathbb C^1_\G(\G, \R^k).$

%

Let $\mathcal F\subset \mathcal O'$ be a closed set such that $S\subset \mathcal F$. Let $f: \mathcal F \to \R^k$ and $g: \mathcal F \to \mathbf M_{k \times m_1}$  be given by
\begin{equation*}
f(Q):=0, \qquad g(Q):=\left(\,\, \mathcal{I}_k \,\, | \,\,  - \nabla^\phi \phi (\Phi ^{-1}( Q)) \, \right)
\end{equation*}
for all $Q\in \mathcal F$, where $\mathcal{I}_k$ is the $k\times k$ identity matrix and $\nabla^\phi \phi (\Phi ^{-1}( Q))$ is the unique $k\times (m_1-k)$ matrix associated to the intrinsic differential $d\phi _{(\Phi ^{-1}( Q))} $ of $\phi$ at $\Phi ^{-1}( Q)$. 

For  any  $\mathcal{K}$ compact in $\mathcal F$, $Q,Q' \in \mathcal{K}$ and $\delta>0$ let 
\begin{equation*}
\rho _\mathcal{K}(\delta ):= \sup_{0< \|Q^{-1} Q' \| <\delta} \frac{ \left\vert g(Q) (Q^{-1} Q')^1 \right\vert_{\R^k}}{\|Q^{-1} Q' \|}.
\end{equation*}
Whenever $Q=A\cdot \phi(A)$ and $Q'=B\cdot \phi(B)$, because the homogeneous norm on the first layer is exactly the Euclidean norm we have
\[
 \left\vert g(Q) (Q^{-1} Q')^1 \right\vert_{\R^k}=\norm{\phi (B)- \phi (A) -  \nabla^\phi \phi (A) (A^{-1} B)^1}
\]
 and, from \eqref{c_0} and the fact that $\W$ is a normal subgroup,
\[
c_0 \norm{\phi(A)^{-1}A^{-1}B\phi(A) }=c_0 \norm{\mathbf P_\W(Q^{-1}\cdot Q')}\leq \norm{Q^{-1}\cdot Q'}.
\]
Hence for any $A_0\in \Phi ^{-1}(\mathcal{K})$
\begin{equation*}
\begin{aligned}
&\frac{ \left\vert g(Q) (Q^{-1} Q')^1 \right\vert_{\R^k}}{\|Q^{-1} Q' \|}\\
 &  \leq \frac 1{c_0}\, \frac{\|\phi (B)- \phi (A) -  \nabla^\phi \phi (A) (A^{-1} B)^1 \|}{\|\phi(A)^{-1}A^{-1}B\phi(A)  \|  } \\
& \leq \frac 1{c_0}\, \left( \frac{\|\phi (B)- \phi (A) -  \nabla^\phi \phi (A_0) (A^{-1} B)^1 \|}{\|\phi(A)^{-1}A^{-1}B\phi(A)  \|  }  + \frac{ \| (\nabla^\phi \phi (A_0)- \nabla^\phi \phi (A)) (A^{-1} B)^1 \|}{\|\phi(A)^{-1}A^{-1}B\phi(A)  \|  }  \right) \\
& \leq \frac 1{c_0}\,  \left( \frac{\|\phi (B)- \phi (A) -  \nabla^\phi \phi (A_0) (A^{-1} B)^1 \|}{\|\phi(A)^{-1}A^{-1}B\phi(A)  \|  }  +  \| \nabla^\phi \phi (A_0)- \nabla^\phi \phi (A)\|  \right) \\
\end{aligned}
\end{equation*}
By the uniformly intrinsic differentiability of $\phi$ in $A_0 \in \mathcal O'$, for all $A, B\in  \Phi^{-1} (\mathcal{K})$ 
\begin{equation*}
\lim_{\delta \to 0^+} \sup_{\|A_0^{-1}A\|<\delta } \sup _{0<\|A^{-1}B\|<\delta} \frac{\|\phi (B)- \phi (A) -  \nabla^\phi \phi (A_0) (A^{-1} B)^1 \|}{\|\phi(A)^{-1}A^{-1}B\phi(A)  \|  }   = 0
\end{equation*}
and, consequently using also the continuity of the intrinsic differential $d\phi _{(\Phi ^{-1}( Q))} $ (see Proposition \ref{prop3.6cont} (3)) we have
\begin{equation*}
\lim_{ \delta \to 0^+} \rho _\mathcal{K}(\delta ) =0.
\end{equation*}
Hence it is possible to apply Theorem $\ref{white}$ and we obtain the existence of   $\hat f \in \C^1_\G (\G ,\R^k)$ such that, for all $Q \in  \mathcal F$
\[
\hat f(Q)= f(Q)= 0 
\]
\[
\nabla_\G \hat f (Q)=g(Q)= \big(\,\, \mathcal{I}_k \,\, | \,\,  -\nabla^\phi \phi (\Phi ^{-1}( Q)) \, \big)
\]
and, in particular, rank$\nabla_\G \hat f (Q) =k$ for all $Q \in \mathcal F$. 

Let $P=A\phi (A)\in S$ and $S':=\{ Q\in \G : \hat f(Q)=0 $ and rank$( \nabla_\G \hat f(Q))=k  \}$. We  just prove that there is a neighbourhood $\mathcal{U}$ of $P$ such that $\mathcal{U} \cap S \subset S'$ and rank$\nabla_\G \hat f (Q) =k$ for all $Q \in S$. 

 In order to get the implication $(2) \Rightarrow (1)$ we need to find a certain neighbourhood $\mathcal{U}(P,r_0)$ of $P$ such that 
\begin{equation}\label{4.13}
\mathcal{U}(P,r_0) \cap S= \mathcal{U}(P,r_0) \cap S'.
\end{equation}
As $P \in \mathcal{U} \cap S \subset S'$ we get
\begin{equation*}
\hat f(P)=0, \quad \nabla_\G \hat f (P)=  \big(\,\, \mathcal{I}_k \,\, | \,\, -\nabla^\phi \phi (\Phi ^{-1}( P)) \, \big)
\end{equation*}   
and by Implicit Function Theorem there exist an open neighborhood $\mathcal {U}'$ of $P$ and a continuous function $\phi ' :\mathcal I(A,\delta ')\to \HH$ such that 
\begin{equation*}
\Phi ' :\mathcal I (A,\delta ') \to S' \cap \overline{\mathcal {U'}}
\end{equation*}
is an homeomorphism given by
\begin{equation*}
\Phi ' (B)= B \phi ' (B), \quad \mbox{for all } B \in \mathcal I (A,\delta ').
\end{equation*}

Moreover $\Phi ^{'-1} (S' \cap \overline{\mathcal {U'}})$ is an open subset of $\mathcal I (A,\delta ')$, with $A \in \Phi ^{'-1} (S' \cap \overline{\mathcal {U'}})$ because $P \in S'\cap \mathcal {U'}$. So there is $\delta '' \in (0,\delta ')$ such that $\mathcal I(A, \delta '') \subset \Phi ^{'-1} (S' \cap \mathcal {U'})$ and, by the uniqueness of the parametrization, we obtain that $\Phi '  \equiv \Phi $ on $\mathcal I(A, \delta '')$.

Now, set $\mathcal {U''}$ and $\mathcal {U'''}$ be an open neighborhood of $P$ in $\G$ such that 
\begin{equation*}
S \cap \mathcal {U''} = \Phi (\mathcal I(A, \delta ''))= \Phi '(\mathcal I (A, \delta ''))= S' \cap \mathcal {U'''}
\end{equation*}
and set $r_0>0$ be such that $ \mathcal{U}(P,r_0) \subset \mathcal {U''} \cap \mathcal {U'''}$. Therefore, since the last equality holds, we observe that $ \mathcal{U}(P,r_0) \cap S= \mathcal{U}(P,r_0) \cap S' $, i.e. $\eqref{4.13}$ holds.

This completes the proof of the theorem.
\end{proof}

We notice that the proof of $(1) \Rightarrow (2)$ and Implicit Function Theorem give the following general fact:
\begin{theorem}\label{theoremNUOVO1}
Let $S$ be a $k$-codimensional $\G$-regular surface and $\mathcal U$ be a neighbourhood  of $P\in S$. We have, for $f\in \C_\G^1(  \mathcal U; \R^k)$, that
\begin{equation*}
S\cap \mathcal U=\{ Q\in \mathcal U \, : \, f(Q)=0\},
\end{equation*}
and  $\W= ker\left(d_\mathbf Pf(P)\right)$ has a complementary horizontal $k$-dimensional subgroup $\HH$ in $\G$. Then there exist $\mathcal O$ open in $\W$, $\phi :\mathcal O  \to \HH$ u.i.d. and a smaller  neighbourhood  of $P$, $\mathcal U' \subset \mathcal U$,  such that
\begin{equation*}
S\cap \mathcal U' =\graph \phi \cap \mathcal U'.
\end{equation*}
\end{theorem}

\begin{proof}
The existence of $\phi:\mcal O \subset \W\to \V$ follows from Implicit Function Theorem. Moreover $\phi$ is uniformly intrinsic differentiable in $\mcal O$ by Theorem \ref{teo4.1}. By the definition of $\W$ and $\HH$ we have that $d_\mathbf Pf(P) :\HH \to \R^k$ is bijective and by continuity this holds also in some neighbourhood of $P$. The same holds for the map $Q \mapsto (d_\mathbf Pf(Q)_{|\HH})^{-1}$ because the map $Q \mapsto d_\mathbf Pf(Q)_{|\HH}$ is a linear isomorphism between $\HH$ and $\R^k$ which depends continuously on $Q$ in a proper neighbourhood of $P$. Then we can argue as in $(1) \Rightarrow (2)$ of Theorem \ref{teo4.1} and we obtain the thesis.
\end{proof}

As a corollary of Theorem $\ref{teo4.1}$ and Theorem $\ref{biblioKOZHEVNIKOV}$, we get a comparison between the Reifenberg vanishing flat set and the uniformly intrinsic differentiable map:
\begin{coroll}
Let $\W$ and $\HH$ be complementary subgroups of a Carnot group $\G$ with $\HH$  horizontal and $k$ dimensional. 
Let $\mathcal O$ be open in $\W$, $\phi :\mathcal O \subset \W \to \HH $ and $S:= \graph{\phi}$. 
Moreover we assume that there is a family $\{ \W_P\, :\, P\in S\}$ of vertical subgroup complementary to $\HH \, ($hence of codimension $k)$ and for every relatively compact subset $S'\Subset S$ there is an increasing function $\beta :(0,\infty ) \to (0,\infty ), \beta (t)\to 0^+$ when $t\to 0^+,$ such that 
\begin{equation*}
\mbox{dist$_d$} \left( \mcal U(P,r)\cap S, \mcal U(P,r)\cap (P\cdot \W_P) \right) \leq \beta (r)r, \quad r>0
\end{equation*}
for any $P\in S'$. Then,  for every $A_0\in \mathcal O$,
\begin{enumerate}
\item there are a neighbourhood $ \mathcal U$ of $A_0\cdot \phi (A_0)$ and $f\in \C_\G^1(  \mathcal U; \R^k)$ such that 
\begin{equation*}
\begin{split}
& S \cap  \mathcal U=\{P\in  \mathcal  U: f(P)=0\}\\
& d_{\bf P}f(Q)_{\vert \HH}:\HH\to \R^k\quad \text{is bijective for all $Q\in \mathcal U$}
\end{split}
\end{equation*}
and $Q\mapsto \left(d_{\bf P}f(Q)_{\vert \HH}\right)^{-1}$ is continuous. 
\item $\phi $ is u.i.d. in a neighbourhood $\mathcal O ' \subset \mathcal O$ of $A_0$. 
\end{enumerate}

\end{coroll} 

\begin{proof}
{$\mathbf {(1)}$} Theorem $\ref{biblioKOZHEVNIKOV}$ states there are a neighbourhood $ \mathcal U$ of $A_0\cdot \phi (A_0)$ and $f\in \C_\G^1(  \mathcal U; \R^k)$ such that $S \cap  \mathcal U=\{P\in  \mathcal  U: f(P)=0\}.$ Using again Theorem $\ref{biblioKOZHEVNIKOV}$ we know that $\W_P=ker(d_\mathbf Pf(P))$ and so the thesis follows from the fact that $\W_P$ is complementary subgroup of $\HH$. 

The condition {$\mathbf {(2)}$} is equivalent to the condition (1) thanks to Theorem \ref{teo4.1}.
\end{proof}

\begin{coroll}
Under the same assumptions of Theorem $\ref{teo4.1}$, if $S:= \graph{\phi}$ satisfies the condition $(1)$ of Theorem $\ref{teo4.1}$, then
 \begin{enumerate}
\item the  function $B\mapsto d\phi(B)$ is continuous in $\mathcal O$.
\item  $\displaystyle \phi \in h^{1/\kappa }_{loc}( \mathcal O )$, that is
$\phi\in \C(\mathcal{O},\R )$ and for all $\mathcal F\Subset \mathcal O$ and $A, B\in \mathcal{F}$
\begin{equation*}
\lim_{r\to 0^+} \sup_{0<\|A^{-1} B\| <r}  \, \frac{\|\phi (B)-\phi (A)\|}{\|A^{ -1}B \|^{1/\kappa}}  =0.
\end{equation*}

\end{enumerate}

\end{coroll}

Observe that u.i.d. functions do exist. In particular, when $\HH$ is a horizontal subgroup,  $\HH$ valued Euclidean $\C^1$ functions are u.i.d.

\begin{theorem}
If $\W$ and $\HH$ are complementary subgroups of a Carnot group $\G$ with $\HH$  horizontal and $k$ dimensional. If
 $\mathcal O$ is open in $\W$ and  $\phi :\mathcal O \subset \W \to \HH $ is such that $\phi  \in \C^1( \mathcal O, \HH)$ then $\phi$ is u.i.d. in $ \mathcal O$. 
\end{theorem}

\begin{proof}
Assume without loss of generality that $X_1,\dots,X_k$ are horizontal vector fields such that $\HH=\exp ({\rm span}\{X_1,\dots,X_k\})$. Then $\phi=\exp \sum_{i=1}^k\phi_i X_i$, where $\phi_i:\mathcal O\to \R$ are $\C^1$ functions.  Let $\mcal U:= \mcal O\cdot \HH$. If $P\in \mcal U$ then $P=P_\W\cdot \exp \sum_{i=1}^k x_i X_i$ with $P_\W\in \mathcal O$ and we define $f:\mathcal U \to \R^k$ as 
\[
f(P)=(x_1-\phi _1(P_\W), \dots, x_k-\phi _k(P_\W)).
\]
With this definition $f\in \C^1(\mathcal U,\R^k)$, hence $f\in \C^1_\G(\mathcal U,\R^k)$. Moreover 
\begin{equation*}
X_l f_j (P) = \frac{d}{ds} f_j (P\exp (sX_l)) _{|s=0}
= \left\{
\begin{array}{l}
1  \quad \mbox{if } j= l \\ 0 \quad \mbox{if } j\ne l
\end{array}
\right.
\end{equation*}
hence ${\rm rank}( \nabla_\G f(P))=k$.

By construction 
$\{Q\in \mathcal U:  f(Q)=0\}= \graph \phi$
then $\graph \phi$ is a non critical level set of $f\in \C^1_\G(\mathcal U,\R^k)$ hence is a $\G$-regular surface. 
From Theorem \ref{teo4.1} it follows that $\phi$ is u.i.d. in $\mathcal O $.
\end{proof}

\section{1-Codimensional Intrinsic graphs in  Carnot groups of step 2}

In this section we characterize uniformly intrinsic differentiable functions $\phi : \mathcal O \subset \W \to \V$, when $\V$ is one dimensional and horizontal, in terms of existence and continuity of suitable  intrinsic derivatives of $\phi$.  Intrinsic derivatives are first order non linear differential operators depending on the structure of the ambient space $\G$ and on the two complementary subgroups $\W$ and $\V$.

In order to do this we have to restrict the ambient space $\G$ under consideration to a subclass of Carnot groups of step two. These groups, denoted here as groups of class $\mathcal B$, are described  in the next subsection where we follow the notations of Chapter 3 of \cite{biblio3}.

\subsection{Carnot groups of class $\mathcal B$} 

\begin{defi}\label{def5.1.1} We say that $\G := (\R^{m+n}, \cdot  , \delta_\lambda )$ is a Carnot group of class $\mathcal B$ if 
there are $n$ linearly independent, skew-symmetric $m\times m$ real matrices $\mathcal{B}^{(1)}, \dots , \mathcal{B}^{(n)}$ such that  
 for all $P=(P^1,P^2)$ and $Q= (Q^1,Q^2)\in \R^{m} \times \R^{n} $ and for all $\lambda >0$
\begin{equation}\label{1}
P\cdot Q = (P^1+Q^1 , P^2+Q^2+ \frac{1}{2}  \langle \mathcal{B}P^1,Q^1 \rangle )
\end{equation}
where $\langle \mathcal{B}P^1,Q^1 \rangle := (\langle \mathcal{B}^{(1)}P^1,Q^1 \rangle, \dots , \langle \mathcal{B}^{(n)}P^1,Q^1 \rangle)$ 
and $\langle \cdot , \cdot \rangle$ is the inner product in $\R^m$ and 
\begin{equation*}
\delta_\lambda P  := (\lambda P^1 , \lambda^2 P^2).
\end{equation*}
Under these assumptions $\G$ is a Carnot group of step 2 with $\R^m$  the  horizontal layer  and $\R^n$  the vertical layer.
\end{defi}
We recall also that, by Proposition 3.4. \cite{biblio2}, for any homogeneous norm in $\G$ there is $c_1>1$ such that for all $P=(P^1,P^2)\in \G$
\begin{equation}\label{deps}
c_1^{-1}\left(\vert P^1\vert_{\R^m}+\vert P^2\vert_{\R^n}^{1/2} \right)\leq\Vert P\Vert\leq c_1\left(\vert P^1\vert_{\R^m}+\vert P^2\vert_{\R^n}^{1/2} \right)
\end{equation}

From now on we will depart slightly from the notations of the previous sections. Precisely, instead of writing $P=(p_1,\dots, p_{m+n})$ we will write 
\[
P=(x_1,\dots,x_m,y_1,\dots, y_n).
\]
With this notation, when $\mathcal{B}^{(s)}:=(b_{ij}^{s})_{i, j=1}^m$,  a basis of  the Lie algebra $\mathfrak g$ of $\G$, is given by the $m+n$ left invariant vector fields
\begin{equation*}
X_j (P) = \partial _{x_j }  +\frac{1}{2 } \sum_ {s=1 }^{n} \sum_ {i=1 }^{m} b_ {ji }^{s} x_i   \partial _{y_s },  \qquad\qquad
Y_s(P)  = \partial _{y_s } ,
\end{equation*}
where $j=1,\dots ,m, $ and $s=1,\dots , n.$


\begin{rem} 
The space of skew-symmetric $m\times m$ matrices has dimension $\frac{m(m-1)}{2}$. Hence in any group $\G$ of class $\mathcal B$ the dimensions $m$ of the horizontal layer and $n$ of the vertical layer are related by the inequality 
\[
n \leq \frac{m(m-1)}{2}.
\]
\end{rem}

\begin{rem}
Heisenberg groups are groups of class $\B$. Indeed $\mathbb H^k= \R^{2k}\times \R$ and 
the group law  is of the form $\eqref{1}$ with  
 \[
\mathcal{B}^{(1)}=  \begin{pmatrix}
0 &  \mathcal{I}_k\\
-\mathcal{I}_k & 0
\end{pmatrix}
\]
where $\mathcal I_k$ is the $k\times k$ identity matrix.

More generally, H-type groups are examples of groups of class $\B$ (see Definition 3.6.1 and Remark 3.6.7 in \cite{biblio3}). In this case  $ \G= \R^m\times\R^{n}$ with 
\begin{equation*}
n<8p+q ,\quad \mbox{with } m = \text{(odd)\,} 2^{4p+q}\, \mbox{ and } \, 0\leq q\leq 3.
\end{equation*}
Observe that if $m$ is odd then $n=0$, hence in the non trivial cases $m$ is even.
The composition law is of the form $\eqref{1}$ where the matrices $\mathcal{B}^{(1)},\dots ,  \mathcal{B}^{(n)}$ have the following additional properties:
\begin{enumerate}
\item $ \mathcal{B}^{(s)}$ is an $m\times m $ orthogonal matrix for all $s=1,\dots n$
\item  $ \mathcal{B}^{(s)}\mathcal{B}^{(l)}=- \mathcal{B}^{(l)} \mathcal{B}^{(s)}$ for every $s,l=1,\dots ,n$ with $s\ne l$.
\end{enumerate}
 Another example of Carnot groups of class $\mathcal B$ is provided by the class $\mathbb F _{m,2}$ of free groups of step-2  (see Section 3.3 in \cite{biblio3}). Here $\mathbb F _{m,2}= \R^{m}\times \R^{\frac{m(m-1)}{2}}$ and  the composition law $\eqref{1}$ is defined by the matrices  $\mathcal B^{(s)}\equiv \mathcal{B}^{(i,j)}$ where $1\leq j< i\leq m$ and $\mathcal{B}^{(i,j)}$ has entries $-1$ in  position $(i,j)$, $1$ in  position $(j,i)$ and $0$ everywhere else. 

Notice that Heisenberg groups are H-type groups while $\mathbb H^1$ is also a free step-$2$ group.
\end{rem}

\subsection{The intrinsic gradient}

 
Let $\G = (\R ^{m+n}, \cdot  , \delta_\lambda )$ be a group of class $\B$ as in Definition \ref{def5.1.1} and $\W$, $\V$ be complementary subgroups in $\G$ with $\V$ horizontal and one dimensional.
 
%

\begin{rem}\label{remIMPORT} To keep notations simpler,
through all this section we assume, without loss of ge\-ne\-ra\-li\-ty, that  the complementary subgroups $\W$, $\V$ are
\begin{equation}\label{5.2.0}
\V:=\{ (x_1,0\dots , 0) \}, \qquad \W:=\{ (0,x_2,\dots , x_{m+n}) \}.
\end{equation}
This amounts simply to a linear change of variables  in the first layer of the algebra $\mathfrak g$.   If we denote $\mathcal{M}$ a non singular $m\times m$ matrix, the linear change of coordinates associated to $\mathcal M$ is 
\begin{equation*}
P=(P^1,P^2)\mapsto (\mathcal M P^1, P^2)
\end{equation*}
The new composition law $\star$ in $\R^{m+n}$,   obtained by writing $\cdot $ in the new coordinates, is
\begin{equation*}
(\mathcal M P^1, P^2) \star (\mathcal M Q^1, Q^2) := (\mathcal M P^1 +\mathcal M Q^1, P^2 + Q^2 + \frac{1}{2} \langle  \mathcal{ \tilde B} \mathcal M P^1 ,\mathcal M Q^1 \rangle),
\end{equation*}
where $\mathcal{\tilde B} := (\mathcal{\tilde B}^{(1)}, \dots , \mathcal{\tilde B}^{(n)})$ 
and 
$ \mathcal{ \tilde B} ^{(s)}= (\mathcal{M}^{-1})^T \mathcal{B}^{(s)}  \mathcal{M}^{-1}$ for $s=1,\dots ,n.$
It is easy to check that the matrices $\mathcal{ \tilde B}^{(1)} ,\dots ,\mathcal{ \tilde B}^{(n)}$ are skew-symmetric and that $(\R^{m+n}, \star, \delta_\lambda)$ is a Carnot group of class $\mathcal B$  isomorphic to $\G = (\R^{m+n}, \cdot, \delta_\lambda)$. 
\end{rem}

When $\V$ and $\W$ are defined as in \eqref{5.2.0} there is a natural inclusion $i: \R^{m+n-1}\to \W$ such that, for all $(x_2,\dots x_m,y_1,\dots,y_n)\in \R^{m+n-1}$,  
\[
i((x_2,\dots x_m,y_1,\dots,y_n)):=(0, x_2,\dots x_m,y_1,\dots,y_n)\in \W.
\]
If $\mathcal O$ and $\phi$ are respectively an open set in $\R^{m+n-1}$ and a function $\phi:\mathcal O\to \R$ 
we denote  $\hat {\mathcal O}:= i(\mathcal O)\subset \W$ and $\hat \phi :\hat{\mathcal O}\to \V$ the function defined as
\begin{equation}\label{phipsi}
\hat\phi (i(A)) :=(\phi(A), 0,\dots ,0) 
\end{equation}
for all $A\in \mathcal O$.
From \eqref{teo4.1.1} in Theorem $\ref{teo4.1}$, if $\hat \phi :\hat{\mathcal O} \subset \W \to \V$ is such that $\graph {\hat \phi} $ is 
locally a non critical level set of  $f\in \C^1_\G (\G, \R)$ with $X_1f\neq 0$, then $\hat\phi$ is u.i.d. in $\hat{\mathcal O}$ and  the following representation of the  intrinsic gradient $\nabla^{\hat \phi} \hat\phi$ holds
\begin{equation}\label{DPHI2}
\nabla^{\hat \phi} \hat\phi (P)=-\left (\frac{X_2f}{X_1f} ,\dots , \frac{X_{m}f}{X_1f}\right)(P\cdot \hat\phi (P))
\end{equation}
for all $P\in \hat{\mathcal O}$.

Moreover, there is an explicit expression of the intrinsic gradient of $\hat \phi $, not involving  $f$, but only derivatives of the real valued function  $\phi$. In Proposition \ref{prop2.22} (see Proposition 4.6 in \cite{biblio1} in Heisenberg groups) we motivate the following definition:

\begin{defi}\label{defintder}
Let $\mcal O $ be  open in $\R^{m+n-1}$,  $\psi :\mcal O \to \R$ be continuous in $\mathcal O$. The \emph{intrinsic derivatives}  $D^\psi _j $, for $j=2,\dots ,m$, are the differential operators with continuous coefficients
\begin{equation*}
\begin{split}
D^\psi _j  &:=\partial _{x_j }+\sum_ {s=1 }^{n} \left(\psi b_ {j 1}^{s} +\frac{1}{2 } \sum_ {l=2 }^{m}  x_l b_ {jl }^{s} \right) \partial _{y_s } \\
&= {X_j}_{\vert \W}+ \psi\,\sum_ {s=1 }^{n}b_ {j 1}^{s} {Y_s}_{\vert \W}
\end{split}
\end{equation*}
where, in the second line with abuse of notation, we denote with the same symbols $X_j$ and $Y_s$ the vector fields acting on functions defined in $\mathcal O$.

If  $\hat\psi:=(\psi, 0,\dots, 0):\hat{\mathcal O}\to \V$, we denote
\emph{intrinsic horizontal gradient} $\nabla^{\hat\psi}$ the differential operator 
\[
\nabla^{\hat\psi} :=(D^\psi _2,\dots, D^\psi _m).
\]
\end{defi}


\begin{prop}\label{prop2.22}
Let $\G := (\R^{m+n}, \cdot  , \delta_\lambda )$ be a Carnot group of class $\B$ and $\V$, $\W$ the complementary subgroups defined in \eqref{5.2.0}. Let $\mcal U$ be  open in $\G$,  $f\in \mathbb{C}^1_\G(\mcal U, \R )$ with $X_1f>0$ and assume that $S:=\{ P\in \mcal U : f(P)=0\}$ is non empty. Then
\begin{enumerate}
\item [(i)] for every $P\in \mcal U$ there exist $\mcal U'$ open neighbourhood of $P$ in $\G, \hat {\mcal O}$ open in $\W$ and $\hat \psi : \hat {\mcal O} \to \V$ such that $S\cap \mcal U' =\graph {\hat \psi} \cap \mcal U'$.  Moreover $\hat \psi$ is u.i.d. in $\hat{\mcal O}$ and the distributional intrinsic gradient $\left(D_2^{\psi} {\psi},\dots, D_m^{\psi} {\psi} \right)$ of the associated function $\psi : \mcal O \to \R$ is defined as
\begin{equation}\label{pde5}
D_j^{\psi} {\psi} = X_j \psi+\psi  \sum_ {s=1 }^{n} b_ {j1 }^{s} Y_s\psi , \quad \mbox{for }\,\, j=2,\dots , m.
\end{equation}
Moreover the distributional intrinsic gradient of $\psi$ has a continuous representative, which is $\nabla^{\hat \psi} \hat\psi \, ($continuous because of \eqref{DPHI2}$)$, i.e. 
\[
\nabla^{\hat \psi} \hat\psi (i(\cdot))=\left(D_2^{\psi} {\psi}(\cdot ),\dots, D_m^{\psi} {\psi}(\cdot) \right)\]
holds in the sense of distribution.
\item [(ii)]
The subgraph $\mcal E:=\{ P\in \mcal U : f(P)<0\}$ has locally finite $\G$-perimeter in $\mcal U$ and its $\G$-perimeter measure $|\partial \mcal E|_\G$ has the integral representation
\[
|\partial \mcal E|_\G (\mathcal{F}) =\int _{\Phi ^{-1}(\mathcal{F})} \sqrt{1+  |\nabla^{\hat \psi} \hat\psi|_{\R^{m-1}}^2} \, \, d\mathcal{L}^{m+n-1}
\]
for every Borel set $\mathcal{F} \subset \mathcal U $ where $\Phi:\mathcal O\to \G$ is defined as $\Phi(A):=i(A)\cdot\hat \psi(i(A))$ for all $A\in\mathcal O$. 
\end{enumerate}
\end{prop} 
\begin{proof} 
The first part of statement (i) follows from Implicit Function Theorem and Theorem \ref{theoremNUOVO1}. 
 Moreover $\hat \psi$ is uniformly intrinsic differentiable in $\hat{\mcal O}$ by Theorem \ref{teo4.1}. Let $\psi :\mathcal O\subset \R^{m+n-1} \to \R$ be the real valued continuous function associated to $\hat \psi$ as in $\eqref{phipsi}$.

Fix $A\in \mcal O $. Let $P:=\Phi (A)$ where $\Phi $ is the  graph map of $\hat \psi$ defined as $\Phi (A):=i(A)\cdot \hat \psi(i(A))$. Because $S:=\graph {\hat \psi} $ is a $\G$-regular hypersurface we know that  there are $r>0$, $\delta>0$ and $f\in \C^1_\G ( \mcal U(P,r) , \R )$ such that $f\circ \Phi =0$ in $\mathcal I(A,\delta)$.

Now we use some results proved in Implicit Function Theorem in  \cite{biblio7} (see Theorem 2.1 in \cite{biblio7}). Arguing as in Step 1 of this theorem we can prove the existence of  $0<r'<r$ and of  a family $(f_\epsilon )_{\epsilon >0} \subset \C^1({ \mcal U(P,r')}, \R)$ such that
\begin{equation}\label{alDpsicDDD1}
\begin{aligned}
f_\epsilon \to f \, & \mbox{  and  } \, \nabla _\G f_\epsilon \to \nabla _\G f \qquad \mbox{  uniformly on } { \mcal U(P,r')} , \, \, \mbox{ as } \, \epsilon \to 0.\\
\end{aligned}
\end{equation}
Moreover as in Step 3 of Implicit Function Theorem in \cite{biblio7} there is $(\psi _\epsilon )_{\epsilon >0} \subset \C^1 (\mathcal I(A,\delta) , \R )$ satisfying
\begin{equation}\label{alDpsicDDD2}
\begin{aligned}
\psi _\epsilon \to \psi \,&  \mbox{  and  } \, - \frac{\widehat{\nabla _\G }f   _\epsilon  }{  X_1 f_\epsilon  } \circ \Phi _\epsilon \to - \frac{ \widehat{\nabla _\G }f  }{  X_1 f  } \circ \Phi 
\end{aligned}
\end{equation}
uniformly on  $\mathcal I(A,\delta)$ as $\epsilon \to 0$, 
where $ \widehat{\nabla _\G }f :=(X_2f,\dots ,X_{m}f )$ and the graph maps $\Phi _\epsilon $ of $\hat \psi _\epsilon =(\psi _\epsilon ,0,\dots ,0)$  are such that $f_\epsilon \circ \Phi _\epsilon  \equiv 0$.\\ 
Moreover, the set $S_\epsilon :=\{ Q\in { \mcal U(P,r')} \, : \, f_\epsilon (Q)=0 \} \supset \Phi _\epsilon (\mathcal I(A,\delta))$ is an Euclidean $\C^1$ surface. Hence, because $f_\epsilon \in \C^1({ \mcal U(P,r')})\subset \C^1_\G({ \mcal U(P,r')})$ and because of Theorem \ref{teo4.1} we have that  $\hat \psi _\epsilon $ (i.e. the parametrization of $S_\epsilon$) is uniformly intrinsic differentiable and by \eqref{DPHI2}, \eqref{alDpsicDDD1} and \eqref{alDpsicDDD2} we get
\[
\nabla ^{\hat \psi _\epsilon } \hat \psi_\epsilon  (i(\cdot )) = - \frac{ \widehat{\nabla _\G }f _\epsilon  }{  X_1 f_\epsilon  } \circ \Phi _\epsilon (\cdot ) \to \nabla ^{\hat \psi  } \hat \psi (i(\cdot ))
\] 
uniformly on  $\mathcal I(A,\delta)$ as $\epsilon \to 0$.

Differentiating  the equality $f_\epsilon \big( i(B)\cdot \hat \psi_\eps(i(B)) \big)= f_\epsilon(\Phi _\epsilon (B)) = 0$, for $B=(x_2,\dots, x_m, y_1,\dots, y_n)\in \mathcal I(A,\delta)$, we get
\[
\begin{array}{lr}
\displaystyle\partial _{x_j} \psi _\epsilon(B)  = -\frac{\partial _{x_j} f_\epsilon(\Phi _\epsilon (B))- \frac{1}{2} \psi_\epsilon(B) \sum_{s=1}^{n} b^{s}_{j1} \partial _{y_s} f_\epsilon(\Phi _\epsilon (B))}{X_1f_\epsilon(\Phi _\epsilon (B))}, &\,  j=2,\dots ,m    \\
 \displaystyle \partial _{y_s} \psi _\epsilon(B)  = -\frac{\partial _{y_s} f_\epsilon(\Phi _\epsilon (B))}{X_1f_\epsilon(\Phi _\epsilon (B))} , & s=1,\dots , n        
\end{array}
\]
Then

\begin{equation*}
\begin{aligned}
& - \frac{X_j f_\epsilon(\Phi _\epsilon (B))}{X_1 f_\epsilon(\Phi _\epsilon (B)) } \\
 & = - \frac{ \partial _{x_j} f_\epsilon (\Phi _\epsilon (B))+ \frac{1}{2} \sum_{s=1}^{n} \left( \psi_\epsilon (B) b^{s}_{j1}+ \sum_{l=2}^{m} x_l b^{s}_{jl}  \right) \partial _{y_s} f_\epsilon(\Phi _\epsilon (B))  } {X_1 f_\epsilon(\Phi _\epsilon (B))}  \\
 & \qquad \mbox{(here we use the skew symmetry of the matrices $\mathcal B^{(s)}$} )\\
  & = - \frac
  { \partial _{x_j} f_\epsilon (\Phi _\epsilon (B))- \frac{1}{2} \psi_\epsilon(B) \sum_{s=1}^{n} b^{s}_{j1} \partial _{y_s} f_\epsilon(\Phi _\epsilon (B))} 
  {X_1 f_\epsilon(\Phi _\epsilon (B))} \\
 &\qquad\quad\quad - \frac{\sum_{s=1}^{n} \left(\psi _\epsilon (B) b^{s}_{j1}+  \frac{1}{2}\sum_{l=2}^{m} x_l b^{s}_{jl}  \right) \partial _{y_s} f_\epsilon(\Phi _\epsilon (B))  } 
  {X_1 f_\epsilon(\Phi _\epsilon (B))}
   \\
&=\partial _{x_j}  \psi_\epsilon (B)+\sum_{s=1}^{n}\left( \psi_\epsilon (B) b^{s}_{j1}+\frac12\sum_ {l=2 }^{m}  x_l  b_ {jl }^{s}\right) \partial _{y_s}  \psi_\epsilon(B).
\end{aligned}  
\end{equation*}

Then, from \eqref{DPHI2} 

\begin{equation*}
D_j^{\psi_\epsilon} {\psi_\epsilon} (B)= \partial _{x_j} \psi_\epsilon(B)  +\sum_ {s=1 }^{n} \left(\psi_\epsilon (B)b_ {j1 }^{s} +\frac{1}{2 } \sum_ {l=2 }^{m}  x_l  b_ {jl}^{s} \right) \partial _{y_s}\psi_\epsilon (B).
\end{equation*}

Letting $\epsilon \to 0^+$ $\eqref{pde5} $ follows from  $\eqref{alDpsicDDD2}$ and so \[
\nabla^{\hat \psi} \hat\psi (i(\cdot))=\left(D_2^{\psi} {\psi}(\cdot ),\dots, D_m^{\psi} {\psi}(\cdot) \right)\]
holds in the sense of distribution. 

Finally using again Implicit Function Theorem in \cite{biblio7}, we know that
\[
|\partial \mathcal E|_\G (\mathcal F) = \int _{\Phi^{-1}(\mathcal F)} \frac{{ |\nabla_\G f(\Phi (B))|_{\R^{m+n-1}}}}{X_1f(\Phi (B))} \, d \mathcal{L}^{m+n-1} (B)
\]
and, consequently, the integral representation of the perimeter $|\partial\mathcal E|_\G$ is true because, from \eqref{DPHI2}, $D^\psi _j\psi = - \frac{X_j f}{X_1 f}  \circ \Phi  $, in the sense of distributions, for $j=2,\dots ,m$. This completes the proof of (ii). 
\end{proof} 

\begin{rem}
We emphasize that the intrinsic horizontal gradient $\nabla^{\hat\psi}$ defined after Definition \ref{defintder} and the continuous representative of the distributional intrinsic gradient as in \eqref{teo4.1.1} are different; but according to the approximation argument of Proposition \ref{prop2.22}, they are equal in distributional sense.
 \end{rem}

In the following proposition we prove that, if $\psi$ is sufficiently regular, the  intrinsic derivatives  $D_j^\psi \psi$ are  the derivatives of $\psi $ along the integral curves of the vector fields $D^\psi _j$. We use a similar technique exploited in Proposition 3.7 in \cite{biblio1} in Heisenberg groups.

\begin{prop}\label{3.7}
Let $\mcal O $ be open in $\W$, $A\in \mcal O$ and  $\hat\psi=(\psi,0,\dots,0) :\mcal O \to \V$. For $j=2,\dots ,m$ and $\delta >0$ denote  $\gamma ^j :[-\delta ,\delta ]\to \mcal O $ an integral curve of the vector field $D^\psi _j$ such that $\gamma ^j(0)=A$. 

If we assume that 
\[
\hat\psi \text{ is u.i.d. in $A\in \mathcal O$ and }
t \mapsto \psi (\gamma ^j(t))
\text{ 
is in $\C^1(-\delta ,\delta)$},
\] 
then 
\[
\lim_{t\to 0} \frac{\psi \bigl(\gamma ^j(t)\bigl)-\psi (A)}{t}= (\nabla^{\hat\psi} \hat \psi)_j(A),
\]
\end{prop}
where with $(\nabla^{\hat\psi} \hat \psi)_j$ we mean the $j$-th component of a representation of the intrinsic differential as in \eqref{teo4.1.1}.
\begin{proof}
Fix $j=2,\dots , m$. Let $A=(x, y)\in \mcal O $ such that $\gamma ^j(0)=A$. We know that $\gamma ^j(t)=\bigl(\gamma _2^j(t),\dots ,\gamma _{m+n}^j(t)\bigl)$ is given by 
\begin{equation}\label{int}
\gamma _h^j(t)=\left\{ 
\begin{array}{lcl}
x_h  &   & h=2,\dots, m\, ,\, h\ne j  \\
x_h +t &   &  h= j \\
 y_h +\frac{1}{2}t\sum_{l=2}^{m} x_l b_{jl}^{h} + b_{j1}^{h} \int_0^t \psi (\gamma ^j(r))\, dr &   &  h=m+1,\dots ,m+ n
\end{array}
\right.
\end{equation}
and then
\begin{equation}\label{ogrande}
\begin{split}
\norm{\hat\psi \bigl(\gamma ^j(t)\bigl)-\hat\psi (A) - d\hat\psi_A \big(A^{-1} \gamma ^j (t)\big)}&=\big|\psi \bigl(\gamma ^j(t)\bigl)-\psi (A) - (\nabla^{\hat\psi} \hat \psi)_j (A) t\big| \\
\norm{\hat\psi(A)^{-1}\cdot A^{-1}\cdot \gamma ^j(t)\cdot \hat\psi(A)}&\leq C |t|.
\end{split}
 \end{equation}
for an appropriate $C>1$. We notice that in the first equality we use the fact that $d\hat\psi_A$ acts only on the first components (see \eqref{teo4.1.1}) and that the homogeneous norm on the first layer is the Euclidean norm. 
Indeed we have
\begin{equation*}
\begin{aligned}
c_1^{-1}\norm{\hat\psi(A)^{-1}\cdot A^{-1}\cdot \gamma ^j(t)\cdot \hat\psi(A)} & \leq |t|+ \sum_{s=1}^{n} \biggl| b_{j1}^{s} \int_0^t \psi (\gamma ^j(r))\, dr - t b_{j1}^{s}  \psi (A) \biggl|^{1/2} \\
& \,\, = |t|+ \sum_{s=1}^{n} \biggl| b_{j1}^{s} \int_0^t \left( \psi (\gamma ^j(r))-  \psi  (A) \right) \, dr  \biggl|^{1/2} 
\end{aligned}
\end{equation*}
Moreover for $h=m+1,\dots ,m+n$, the map $t \mapsto \gamma _h^j(t)$ is of class $\C^2$ (because of $\eqref{int}$ and the hypothesis that $t \mapsto \psi (\gamma ^j(t))$ is $\C^1$), hence for all $s=1,\dots ,n$
\begin{equation*}
\begin{aligned}
b_{j1}^{s}\int_0^t \left( \psi (\gamma ^j(r))-  \psi  (A) \right) \, dr  = O(t^2).
\end{aligned}
\end{equation*}
Hence $\eqref{ogrande}$ holds with a appropriate $C= C(c_1)>0$, and from $\eqref{ogrande}$ we get
\begin{equation*}
\frac{\big|\psi \bigl(\gamma ^j(t)\bigl)-\psi (A) - (\nabla^{\hat\psi} \hat \psi)_j (A) t\big|}{t}\, 
\leq \, C \;\frac{\big|\hat\psi \bigl(\gamma ^j(t)\bigl)-\hat\psi (A) - d\hat\psi_A \big(A^{-1} \gamma ^j (t)\big)\big|}{\Vert{\hat\psi(A)^{-1}\cdot A^{-1}\cdot \gamma ^j(t)\cdot \hat\psi(A)}\Vert}
\end{equation*}
where $d\hat\psi _A$ is the intrinsic differential of $\hat\psi$ at $A$. By letting $t \to 0$ and using the assumption of intrinsic differentiability of $\hat\psi$ at $A$, we obtain the thesis.
\end{proof}

\subsection{Broad$^\ast$ solutions and $\nabla^\psi$-exponential maps}
In this section we prove a converse of the previous results. In sections 3 and  4 we proved that if $\graph {\hat{\psi}}$ is a $\G$-regular hypersurface  then the intrinsic derivatives $D^{\psi}_j\psi$  are continuous.  Here we show that if $\psi:\mathcal O\subset \R^{m+n-1}\to \R$ is a 
continuous solution in an open set $\mathcal O$ of the non linear first order system
\begin{equation}\label{sistema}
\left(D^{\psi}_2\psi,\dots, D^{\psi}_m\psi\right)= w 
\end{equation}
where $w:\mathcal O\to \R^{m-1}$ is a given continuous function then $\hat\psi: \hat{\mathcal O}\to \V$ is uniformly intrinsic differentiable in $\hat{\mathcal O}$ and consequently its graph is a $\G$-regular hypersurface. The main equivalence result is contained in Theorem \ref{finale1}.
 
In order to state the equivalence result we have to be precise about the meaning of being a solution of \eqref{sistema}.  To this aim we recall a  notion of generalized solutions of systems of this kind. These generalized solutions, denoted
\emph{broad* solutions} were introduced and studied for the system \eqref{sistema} inside Heisenberg groups in \cite{biblio1, biblio27}. For a more complete bibliography we refer to the bibliography in \cite{biblio1}.

\begin{defi}\label{defbroad*}
Let $\mcal O \subset  \R^{m+n-1}$ be open and $w:=\left(w_2,\dots,w_m\right):\mcal O \to \R^{m-1}$  a continuous function.
With the notations of Definition \ref{defintder}  
we say that $\psi\in \C(\mcal O, \R)$ is a \emph{broad* solution} in $\mcal O$ of the system
\[
\left(D^{\psi}_2\psi,\dots, D^{\psi}_m\psi\right)= w \]
 if for every $A\in \mcal O$ there are $0< \delta_2 < \delta_1$ and $m-1$ maps $\exp _A(\cdot D^\psi _j)(\cdot)$
\begin{equation*}
\begin{split}
\exp _A(\cdot D^\psi _j)(\cdot):[-\delta _2,\delta _2]&\times  \mathcal I(A,\delta_2)\to \mathcal I(A, \delta_1)\\
(t,&B)    \mapsto  \exp _A(tD^\psi _j)(B )
\end{split}
\end{equation*}
for $j=2,\dots , m$, where  $ \mathcal I(A, \delta):= \mathcal U(A,\delta)\cap \W$ and $\mathcal I(A, \delta_1)\subset\mathcal O$. Moreover these maps, called \emph{exponential maps} of the vector fields $D^\psi _2,\dots,D^\psi _m$, are required to have the following properties 
\[t\mapsto \gamma^j_B (t):=\exp_A (tD^\psi _j)(B) \in \C^1 ([-\delta _2,\delta _2], \R^{m+n-1})\]
for all $B\in  \mathcal I(A,\delta_2)$ and 
\[\left\{
\begin{array}{l}
\dot \gamma ^j_B =D^\psi _j \circ \gamma ^j_B\\
\gamma ^j_B (0)=B
\end{array}
\right.
\]
\[
\psi (\gamma ^j_B (t))-\psi (\gamma ^j_B (0))=\int_0^t w_j (\gamma ^j_B (r))\, dr \quad \mbox{ for } t\in [-\delta_2, \delta_2]
\]
once more 
for all $B\in  \mathcal I(A,\delta_2)$.
\end{defi}

\begin{rem}\label{rem5.4} 
If the exponential maps of $D^\psi _2,\dots,D^\psi _m$ at $A$ exist, then the map
\begin{equation*}
[-\delta _2,\delta _2]\ni t \longmapsto \psi \biggl(\exp _A (tD^\psi _j )(B)\biggl)
\end{equation*}
is of class $\C^1$ for every $j=2,\dots , m$ and for each $B\in \mathcal I_{\delta _2 }(A)$.
\end{rem}


\begin{theorem}\label{finale1}
Let $\G$ be a group of class $\B$ and that 
\[
\V:=\{ (x_1,0\dots , 0) \}, \qquad \W:=\{ (0,x_2,\dots , x_{m+n}) \}.
\]
Let $\mcal O$  open in $\R^{m+n-1}$ and $\psi:\mathcal O\to \R$, $\hat\psi :\hat{\mcal O} \to \V $ be a continuous function. Then the following conditions are equivalent:
\begin{enumerate}
\item $\graph {\hat\psi} $ is a $\G$-regular hypersurface, i.e. for all $A\in \graph{\hat \psi}$ there is $r=r(A)>0$ and $f\in \mathbb C_{\G}^1(\mathcal U(A,r)),\R)$ with $X_1f>0$ such that
\[
\graph{\hat \psi}\cap \mathcal U(A,r)=\{ P:f(P)=0\}
\]
\item  $\hat\psi $ is u.i.d. in $ \hat {\mcal O}$.
\item there exists $w\in \C(\mcal O , \R ^{m-1})$ such that, if we consider the distribution $D_j^{\psi } {\psi } $ defined as in \eqref{pde5}, we have $\left(D^{\psi}_2\psi,\dots, D^{\psi}_m\psi\right)= w $ in the sense of distribution. Moreover, for $0<\eps<1$, there is a family of functions $\psi_\eps\in \mathbb C^1(\mathcal O,\R)$ such that for all $\mathcal O'\Subset \mathcal O$, 
\begin{equation*}
\psi _\epsilon \to \psi  \quad  and \quad  D_j^{{\psi _\epsilon}} {\psi _\epsilon} \to w
\end{equation*}
for $j=2,\dots, m$, uniformly on $\mcal O '$  as $\epsilon \to 0^+$.
\item $\psi  \in h^{1/2}_{loc}( \mcal O )$ $($see Proposition $\ref{prop3.6cont}$ $(2))$ and there exists $w\in \C(\mcal O , \R ^{m-1})$ such that $\psi $ is a broad* solution of 
\[
\left(D^{\psi}_2\psi,\dots, D^{\psi}_m\psi\right)= w 
\]
 in $\mcal O $.
\end{enumerate}
\end{theorem}
\begin{proof}
We split the proof in several steps. We should cite as inspiration Theorem 5.1 in \cite{biblio1} for the proof of $\mathbf {(2) \implies (3)}$; Lemma 5.6 in \cite{biblio1} for the proof of $\mathbf {(3) \implies (4)}$ and Lemma 5.7 in \cite{biblio1} for the proof of $\mathbf {(4) \implies (2)}$. 

{$\mathbf {(1) \iff (2)}$}  see Theorem \ref{teo4.1}.

{$\mathbf {(2)\implies (3)}$}

In the proof of Proposition \ref{prop2.22} (i), we prove that for all $A\in \mcal O $ there are $\delta =\delta (A)>0$ with $\mathcal I(A,\delta) := \mcal U(A,\delta) \cap \mcal O $, a map $w\in \C(\mcal O , \R ^{m-1})$ and a family $(\psi _\epsilon )_{\epsilon >0} \subset \C^1 (\mathcal I(A,\delta) , \R)$ such that
\begin{equation}\label{5.12}
\psi _\epsilon \to \psi \,  \mbox{ and } \, D_j^{{\psi _\epsilon}} {\psi _\epsilon} \to w \qquad \mbox{ uniformly on }\,  \mathcal I(A,\delta), \, \,  \mbox{ as } \, \epsilon \to 0.
\end{equation}
Moreover, Proposition \ref{prop2.22} (i) states that the distribution $D^{{\psi}} {\psi}$ as \eqref{pde5} is equal to $w$ in the distributional sense.
%
%

From what proved up to now we know that  for all $B\in \mcal O $ there is $\delta=\delta (B)>0$ such that  $\mathcal I(B,\delta)\Subset  \mcal O $ and  there is family $\mathcal{A}= \mathcal A(B)$ of $\C^1 $ functions defined as $\mathcal{A}:=\{ \psi _{\epsilon , B} : {\mathcal I(B,\delta)} \to \R \} _{0<\epsilon <1}$ such that $\eqref{5.12}$ holds.

Now let $ \mathcal {\hat F}:= \{ \mbox {Int} (\mathcal I(B,\delta))  :B\in \mcal O \}$  be open covering of $\mcal O $. Then there exists a locally finite countable subcovering $ \mathcal {F}$ of $ \mathcal {\hat F}$ and $\{ \theta _h :\mcal O \to \R \, | \, h\in \N \}$  a partition of the unity  subordinate to $\mathcal{F}$. 

Let  $\psi _{\epsilon , B_h} \in \mathcal{A}$. We consider $\psi _{\epsilon , h}:=\psi _{\epsilon , B_h}:\R^{m+n-1} \to \R$ where from now on, if necessary, we use the convention of extending functions by letting them vanish outside an enlargement of their domain, to preserve the regularity. Let $\psi _\epsilon := \sum_{h=1}^{\infty}\theta_h \psi _{\epsilon , h}$; by construction $\psi _\epsilon \in \C^1 (\mcal O, \R )$ and 
\begin{equation*}
D_j^{ \psi _\epsilon } \psi _\epsilon = \sum_{h=1}^{\infty}\biggl( \psi _{\epsilon , h} D_j^{ \psi _\epsilon } \theta _h + \theta _h  D_j^{ \psi _\epsilon }\psi _{\epsilon , h} \biggl) \quad \mbox {  on } \mcal O 
\end{equation*}
Because the partition is locally finite, there are only a finite number of $h_1,\dots ,h_l$ such that $\overline{\mcal O '}\cap$ spt$\theta _{h_\tau } \ne \emptyset $ for each $\tau =1,\dots ,l$ and $\overline{\mcal O '}\subset \bigcup_{\tau =1}^l $spt$\theta _{h_\tau }$. Then
\begin{equation*}
\psi _\epsilon = \sum_{\tau =1}^{l} \theta _{h_\tau} \psi _{\epsilon, h_\tau} \, \mbox {   and   } \, \psi  = \sum_{\tau =1}^{l} \theta _{h_\tau } \psi \,   \mbox {  on  } \, \overline{\mcal O '}
\end{equation*}
\begin{equation*}
D_j^{ \psi _\epsilon } \psi _\epsilon = \sum_{\tau =1}^{l}\biggl( \psi _{\epsilon , h_\tau} D_j^{ \psi _{\epsilon } } \theta _{h_\tau }+ \theta _{h_\tau }  D_j^{ \psi _{\epsilon  }}\psi _{\epsilon , h_\tau } \biggl) \mbox {  on  } \, \overline{\mcal O '}
\end{equation*}
Putting together the last equalities and $\eqref{5.12}$ we get
\begin{equation*}
\psi _\epsilon \to \psi \quad \mbox {  and  } \quad D_j^{ \psi _\epsilon }  \psi _\epsilon \to \sum_{\tau =1}^{l}\Bigl( \psi D_j^{ \psi  }\theta _{h_\tau } + \theta _{h_\tau } D_j^{ \psi  }\psi \Bigl)= w_j \quad \mbox {  uniformly on  } \overline{\mcal O '}, \, \mbox {  as  } \epsilon \to 0
\end{equation*}
 This completes the proof of the implication $(2)\implies (3)$.

\medskip
$\mathbf{(3)  \implies  (4)}$

The proof of $\psi  \in h^{1/2}_{loc}( \mcal O )$ is the content of Theorem \ref {teo5.9}. 

To prove that $\psi$ is a broad$^\ast$ solution we have to show  that for each $A\in \mcal O $ there exist $\delta _1, \delta_2>0$ with $\delta_1>\delta_2$ such that for $j=2,\dots , m$ there is an exponential map $\exp_A(tD^\psi _j)(B)\in  \mathcal I (A,\delta_1) \Subset  \mcal O $ for all $(t,B)\in [-\delta_2,\delta_2]\times  \mathcal I (A,\delta_2) $; moreover,
\begin{equation*}
w_j(B)=\frac{d}{dt} \psi \big( \exp_A (tD^\psi _j)(B)\big)_{| t=0}
\end{equation*}
for all $B\in \mathcal I (A,\delta_2)$.

Fix $j=2,\dots , m$. For $\epsilon >0$ we consider the Cauchy problem
\begin{equation*}
\left\{
\begin{array}{l}
\dot \gamma ^j_{B, \epsilon } (t)= D^{\psi _\epsilon } _j ( \gamma ^j_{B, \epsilon } (t))\\ \\
\gamma ^j_{B, \epsilon } (0)= B
\end{array}
\right.
\end{equation*}
which has a solution $\gamma_\epsilon :[-\delta _2(\epsilon),\delta _2(\epsilon)]\times  \overline{\mathcal I(A, \delta_2 (\epsilon))}\to \overline{ \mathcal I (A,\delta_1)}$. By Peano's estimate on the existen\-ce time for solutions of ordinary differential equations we get that $\delta_2 (\epsilon)$ is greater than $C/\| \psi _\epsilon \|_{ \mathcal{L}^\infty ( \mathcal I(A, \delta_1 (\epsilon)) )} $, with $C$ depending only on $\delta _1$. So it is sufficient to take $\delta _2>0$ such that $\delta_2 \leq \delta_2 (\epsilon)$ for all $\epsilon$. Because $\gamma_\epsilon$ are uniformly continuous on the compact $[-\delta _2,\delta _2]\times \overline{ \mathcal I(A, \delta_2)}$, by Ascoli-Arzel\'a Theorem, we have a sequence $(\epsilon _h)_h$ such that $\epsilon _h \to 0$ as $h\to \infty$ and $\gamma _{ \epsilon _h}\to \gamma $ uniformly on $[-\delta _2,\delta _2]\times \overline{ \mathcal I(A, \delta_2 )}$. Obviously,
\begin{equation*}
\gamma_{ B, \epsilon _h} ^j(t)= B + \int_0^t D^{\psi _{ \epsilon _h} }_j (\gamma_{ B, \epsilon _h} ^j(r)) \, dr
\end{equation*}
\begin{equation*}
\psi _{ \epsilon _h}(\gamma_{B, \epsilon _h} ^j(t)) - \psi _{ \epsilon _h}(\gamma_{ B, \epsilon _h} ^j(0))= \int_0^t  D^{\psi _{ \epsilon _h} }_j \psi _{ \epsilon _h}(\gamma_{B, \epsilon _h} ^j(r))\, dr
\end{equation*}
and for $h\to \infty$ using that all the involved convergences are uniform we conclude
\begin{equation*}
\gamma ^j_B(t)= B + \int_0^t D^\psi _j (\gamma _B^j(r)) \, dr
\end{equation*}
\begin{equation*}
\psi (\gamma _B^j(t)) - \psi (\gamma _B^j(0))= \int_0^t w_j (\gamma _B^j(r))\, dr
\end{equation*}
i.e. the conditions of the Definition $\ref{defbroad*}$ are satisfied.

%
%

\medskip
$\mathbf{(4)\implies (2)}$.



%

Let us fix $A=(\bar x,\bar y)\in \mcal O $.
First let $B=( x,y),B'= ( x',y') \in \mathcal I(A, \delta )$ for a sufficiently small $\delta >0$.

Here we can not integrate along the vector field $D^\psi_j$; however this obstacle can be solved using the exponential maps, more precisely by posing
\begin{equation*}
B_i  :=\exp _A(\bar D_i)(B_{i-1} ) \quad \mbox{ for  }  i=2,\dots, m
\end{equation*} 
where $\bar D:=(\bar {D}_2,\dots,\bar D_m)$ is the family of vector fields given by $\bar {D}_j=(x'_j-x_j)D^\psi _j$ for $j\in \{2,\dots ,m\}$. A computation gives that 
\begin{equation*}
B_j=(x'_2,\dots ,x'_j,x_{j+1},\dots ,x_m,y^{B_j})
\end{equation*} 
with
\begin{equation*}
\begin{aligned}
y^{B_j}_s & = y_s+\sum_{l=2}^{j} \biggl(- b_{1l}^{s} \int_0^{x'_l-x_l} \psi \Big(\exp _A(rD^\psi _l )(B_{l-1})\Big)\, dr +\frac{1}{2}(x'_l-x_l)\Big(\sum_{ i=2 }^{l} x'_i b_{li}^{s}+\sum_{ i=l+1}^{m} x_i b_{li}^{s}\Big)  \biggl)\\
& = y^{B_{j-1}}_s - b_{1j}^{s} \int_0^{x'_j-x_j} \psi \Big(\exp_A(rD^\psi _j )(B_{j-1})\Big)\, dr +\frac{1}{2}(x'_j-x_j)\Big(\sum_{ i=2 }^{j} x'_i b_{ji}^{s}+\sum_{ i=j+1}^{m} x_i b_{ji}^{s}\Big) 
\end{aligned}
\end{equation*} 
for $s=1,\dots ,n$. Observe that $B_2,\dots ,B_m $ are well defined for a sufficiently small $\delta $.   Because $\psi$ is of class $\C^1$ (see Remark $\ref{rem5.4}$) we have
\begin{equation*}
\begin{aligned}
\psi (B') -\psi (B) & = [\psi (B') -\psi (B_m)]+ \sum_{ l=2 }^{m}[\psi (B_l) -\psi (B_{l-1})]\\
                          & = [\psi (B') -\psi (B_m)]+\sum_{ l=2 }^{m} \left(\bar D_l \psi (B_{l-1}) +o (|x'_l- x_l|)\right) \\
\end{aligned}
\end{equation*}
 Notice that, in the last equality, we used the fact 
\begin{equation*}
\begin{aligned}
\sum_{ l=2 }^{m}\bigl(\psi (B_l) -\psi (B_{l-1})\bigl) &=\sum_{ l=2 }^{m}\biggl(\int _0^1 (x'_l-x_l)w_l (\exp_A (r\bar D_l)(B_{l-1}))\, dr \biggl)\\
 & \qquad \mbox{(by the continuity  of $w$} )\\
& =\sum_{ l=2 }^{m} \left( (x'_l-x_l)w_l (B_{l-1}) +o (|x'_l- x_l|)\right)\\
\end{aligned}
\end{equation*}

Now, using again  the continuity of $w$, we have
\begin{equation*}
\lim_{\delta  \to 0} \frac{ \sum_{ l=2 }^{m} (x'_l-x_l)w_l (B_{l-1}) -\langle w(A),  x'- x \rangle}{|x'- x| _{\R^{m-1}}}=0   
\end{equation*}
and so 
\begin{equation*}
\begin{aligned}
\psi (B') -\psi (B) & = \psi (B') -\psi (B_m)+\langle w(A),  x'- x \rangle+ \left(\sum_{ l=2 }^{m}(x'_l-x_l)w_l (B_{l-1})  -\langle w(A),  x'- x \rangle \right)\\
& \quad  + o (|x'- x|_{\R^{m-1}}) \\
 & = \psi (B') -\psi (B_m)+ \langle w(A),  x'- x \rangle +  o ( \|\hat \psi(B)^{-1}B^{-1}B'\hat \psi(B)  \|). 
\end{aligned}
\end{equation*}
Consequently, it is sufficient to show that 
$
\psi (B') -\psi (B_m) = o ( \|\hat \psi(B)^{-1}B^{-1}B'\hat \psi(B)  \|).
$
First we observe that
\[
\frac {|\psi (B') -\psi (B_m)|}{  \|\hat \psi(B)^{-1}B^{-1}B'\hat \psi(B)  \| } \leq  C_\psi (\delta ) \frac {|y'-y^{B_m}|_{\R^{n}}^{1/2}}{  \|\hat \psi(B)^{-1}B^{-1}B'\hat \psi(B)  \|}
\]

with
\begin{equation*}\label{varpi}
C_\psi (\delta ) := \sup \biggl\{ \frac{|\psi (A') -\psi (A'')|}{\vert A' -A''\vert^{1/2}_{\R^{m+n-1}}  }\, :\, A' \ne A'', \, \, A', A''\in  \mathcal I(A,\delta ) \biggl\} .
\end{equation*}
We know also that $\lim_{\delta \to 0} C_\psi (\delta )=0$ because $\psi  \in h^{1/2}_{loc}( \mcal O )$. 
So it is evident that remains to prove $|y'-y^{B_m}|^{1/2} _{\R^{n}} / \|\hat \psi(B)^{-1}B^{-1}B'\hat \psi(B)  \|$ is bounded in a proper neighborhood of $A$. 

If we put $\mathcal{B}_{M} = \max \{ b_{ij}^{s} \, | \,  i,j=1,\dots , m \, , s=1,\dots , n \}$ (observe that $\mathcal{B}_{M}>0$ because the matrices $\B^{(s)}$ are skew-symmetric), then 
\begin{equation}\label{infinity}
\begin{aligned}
|y'-y^{B_m}|_{\R^{n}}    & \leq \sum_{ s=1}^{n} \biggl|y'_s-y_s+\sum_{l=2}^{m} \biggl( b_{1l}^{s} \int_0^{x'_l-x_l} \psi \Big(\exp_A(rD^\psi _l) (B_{l-1})\Big)\, dr +\\
&\quad -\frac{1}{2}(x'_l-x_l)\Big(\sum_{ i=2 }^{l} x'_i b_{li}^{s}+\sum_{ i=l+1}^{m} x_i b_{li}^{s}\Big)  \biggl) \biggl| \\
& \leq \sum_{ s=1}^{n} \biggl|y'_s- y_s +\psi (B)\sum_{ l=2 }^{m} (x'_l-x_l) b_{1l}^{s} -\frac{1}{2}\langle \mathcal{B}^{(s)} x, x'- x \rangle \biggl|\\
& \quad + \sum_{ s=1}^{n} \biggl| -\frac{1}{2}(x'_l-x_l)\Big(\sum_{ i=2 }^{l} x'_i b_{li}^{s}+\sum_{ i=l+1}^{m} x_i b_{li}^{s}\Big) + \frac{1}{2} \langle \mathcal{B}^{(s)}  x,x'- x \rangle \biggl| \\
& \quad+\sum_{ s=1}^{n} \biggl| -\psi (B)\sum_{ l=2 }^{m} (x'_l-x_l) b_{1l}^{s} +\sum_{l=2}^{m}  b_{1l}^{s} \int_0^{x'_l-x_l} \psi \Big(\exp_A(r D^\psi _l )(B_{l-1})\Big)\, dr \biggl| \\
& \leq c_1 \|\hat \phi(i(B))^{-1}i(B)^{-1}i(B')\hat \phi(i(B))  \|^2 + \frac{1}{2}n\mathcal{B}_M|x'-x|^2_{\R^{m-1}} \\
&  \quad +\sum_{ s=1}^{n} \biggl| -\psi (B)\sum_{ l=2 }^{m} (x'_l-x_l) b_{1l}^{s} 
 +\sum_{l=2}^{m}  b_{1l}^{s} \int_0^{x'_l-x_l} \psi \Big(\exp_A(r D^\psi_l) (B_{l-1})\Big)\, dr \biggl| \\
\end{aligned}
\end{equation}
where $c_1 $ is given by \eqref{deps}.  
Note that we have used
\begin{equation*}
\begin{aligned}
\Big| \frac{1}{2} \langle \mathcal{B}^{(s)} x,  x' & - x\rangle -\frac{1}{2}(x'_l -x_l)\Big(\sum_{ i=2 }^{l} x'_i b_{li}^{s} +\sum_{ i=l+1}^{m} x_i b_{li}^{s}\Big) \Big| \\
&=\left| -\frac{1}{2}(x'_l-x_l)\Big(\sum_{ i=2 }^{l} x'_i b_{li}^{s}+\sum_{ i=l+1}^{m} x_i b_{li}^{s} -\sum_{ i=2}^{m} x_i b_{li}^{s} \Big) \right| 
 \leq \frac{1}{2}\mathcal{B}_M|x'-x|^2_{\R^{m-1}} .
\end{aligned}
\end{equation*}

Finally, we consider the last term of \eqref{infinity}
\begin{equation*}
\begin{aligned}
&  \sum_{ s=1}^{n} \biggl| -\psi (B)\sum_{ l=2 }^{m} (x'_l-x_l) b_{1l}^{s} 
 +\sum_{l=2}^{m}  b_{1l}^{s} \int_0^{x'_l-x_l} \psi \Big(\exp_A(r D^\psi_l )(B_{l-1})\Big)\, dr \biggl| \\
 &\qquad \leq R_1(B,B')+R_2(B,B')
\end{aligned}
\end{equation*}
where
\begin{equation*}
\begin{aligned}
R_1(B,B')& :=  \sum_{ s=1}^{n} \sum_{l=2}^{m}\biggl|   b_{1l}^{s} \int_0^{x'_l-x_l} \psi \Big(\exp_A(rD^\psi _l) (B_{l-1})\Big)\, dr  -  b_{1l}^{s} \psi (B_{l-1})(x'_l-x_l)   \biggl|\\
R_2(B,B') &:=  \sum_{ s=1}^{n} \biggl| \sum_{l=2}^{m} b_{1l}^{s}(x'_l-x_l) \Big( \psi (B_{l-1})-\psi (B)\Big)  \biggl| \\
\end{aligned}
\end{equation*}

We would show that there exist $C_1, C_2 >0$ such that
\begin{equation}\label{r1}
R_1(B,B') \leq C_1 |x'-x|^2_{\R^{m-1}} 
\end{equation}
\begin{equation}\label{r2}
R_2(B,B') \leq C_2 |x'-x|^2_{\R^{m-1}} 
\end{equation}
for all $B,B'\in \mathcal I(A,\delta )$, and consequently \begin{equation*}
|y'-y^{B_m}|_{\R^{n}}   \leq  c_1 \|\hat \phi(i(B))^{-1}i(B)^{-1}i(B')\hat \phi(i(B))  \|^2 + \left( \frac{1}{2} n\mathcal{B}_M   +C_1+C_2 \right) |x'-x|^2_{\R^{m-1}}  
\end{equation*}
Hence there is $ C_3>0$ such that
\begin{equation*}
|y'-y^{B_m}|^{1/2} _{\R^{n}}  \leq  C_3  \|\hat \psi(B)^{-1}B^{-1}B'\hat \psi(B)  \|
\end{equation*}
which is the thesis.

We start to consider $R_1(B,B')$. Fix $l=2,\dots , m$ and $s=1,\dots , n$. For $t\in [-\delta , \delta ]$ 
we define 
\begin{equation*}
g_l^s (t):=b_{1l}^{s} \biggl(\int_0^t \psi (\exp_A (rD^\psi_l)(B_{l-1}))\, dr - t \psi (B_{l-1})\biggl)
\end{equation*}
working as in Proposition $\ref{3.7}$ we  show the existence of $C_l^s>0$ such that
\begin{equation*}
|g_l^s (t)| \leq C_l^s t^2, \hspace{0,5 cm} \forall t\in [-\delta , \delta ]
\end{equation*}
So set $t=x'_l-x_l$ we get
\begin{equation*}
|g_l^s (x'_l-x_l)| \leq C_l^s (x'_l-x_l)^2
\end{equation*}
and consequently $\eqref{r1}$ follows from 
\begin{equation*}
\sum_{l=2}^{m} \sum_{s=1}^{n}  |g_l^s (x'_l-x_l )| \leq \sum_{l=2}^{m} \sum_{s=1}^{n}  C_l^s (x'_l-x_l)^2 \leq   C_1|x'-x|_{\R^{m-1}} ^2.
\end{equation*}

Now we consider $R_2(B,B')$. Observe that
\begin{equation*}
\begin{aligned}
 \sum_{ s=1}^{n} \biggl| \sum_{l=2}^{m} b_{1l}^{s}(x'_l& -x_l) \Big( \psi (B_{l-1})-\psi (B)\Big)  \biggl|  \leq  n \mathcal{B}_M \sum_{l=3}^{m} |x'_l-x_l| \Bigl| \psi (B_{l-1})-\psi (B)  \Bigl| \\
& \leq n \mathcal{B}_M \sum_{l=3}^{m} |x'_l-x_l| \Big(\sum_{i=2}^{l-1}|\psi (B_i)-\psi (B_{i-1})| \Big)\\
& \leq  n \mathcal{B}_M \sum_{l=3}^{m} |x'_l-x_l| \Big( \sum_{i=2}^{l-1}\Big| \int_0^1 (x'_i-x_i)w_i (\exp _A(r\bar D_i)(B_{i-1}))\, dr \Big|\Big) \\
& \leq  n \mathcal{B}_M \sum_{l=3}^{m} |x'_l-x_l|  \Big(\sum_{i=2}^{l-1} \Big|(x'_i-x_i)(w_i( B_{i-1} )+o(1) )  \Big| \Big)\\
& \leq  n \mathcal{B}_M C |x'-x|^2_{\R^{m-1}} 
\end{aligned}
\end{equation*}
Then $\eqref{r2}$ follows with $ C_2:=\frac{1}{2}n \mathcal{B}_M C $. As a consequence the $(4)\implies (2)$  is proved.

\end{proof}

We prove now that  the solutions of the system $\left(D^{\psi}_2\psi,\dots, D^{\psi}_m\psi\right) =w$  when $w\in \C (\mcal O , \R^{m-1})$ are H\"older continuous.  We use the similar technique exploited in Theorem 5.8 and Theorem 5.9  in \cite{biblio1} in the context of Heisenberg groups.

\begin{theorem}\label{teo5.9}
Let $\psi \in \C(\mcal O, \R )$ where $\mcal O $ is open in $\R^{m+n-1}$. Assume  that  there exist $w:= (w_2,\dots , w_m)\in \C(\mcal O, \R^{m-1})$ 
and a family $(\psi _\epsilon)_{\epsilon>0 }\subset \C^1(\mcal O, \R )$ such that, for any open $\mcal O ' \Subset  \mcal O $,
\begin{equation*}
\psi _\epsilon \to \psi  \quad  and \quad  D^{\psi _\epsilon} \psi _\epsilon \to w \quad \text{ uniformly on } \, \, \, \mcal O ' ,\, \, \mbox{ as } \, \epsilon \to 0^+.
\end{equation*}

Then, for $\mcal O'\Subset  \mcal O ''\Subset  \mcal O $ there exists $\alpha :(0,+\infty )\to [0,+\infty )$ depending only on $\mcal O ''$, $\|\psi \|_{\mathcal{L}^\infty (\mcal O '')}$, $\| w \|_{\mathcal{L}^\infty (\mcal O '')}$,  on $\mathcal{B}^{(1)} , \dots , \mathcal{B}^{(n)}$ and on the modulus of continuity of $w$ on $\mcal O ''$ such that  
\begin{equation}\label{limitealpha}
\lim_{r\to 0}\alpha (r)=0
\end{equation}
and
\begin{equation}\label{betafin}
\sup \biggl\{  \frac{|\psi (A)-\psi  (A')|}{\vert A-A'\vert ^{1/2} _{\R^{m+n-1}} } : \, A,A'\in \mcal O ', 0<\vert A-A'\vert  _{\R^{m+n-1}}  \leq r \biggl\} \leq \alpha (r).
\end{equation}
\end{theorem}

\begin{proof}
It is sufficient to prove the theorem for  $\psi \in \C^1 (\mcal O, \R)$. Indeed 
from the uniform convergence of $\psi _\epsilon $ and  $\left(D_2 ^{\psi _\epsilon }\psi _\epsilon,\dots, D_m ^{\psi _\epsilon }\psi _\epsilon \right) $, we can estimate $\|\psi _\epsilon \|_{\mathcal{L}^\infty (\mcal O ')}$, on $\|\left(D_2 ^{\psi _\epsilon }\psi _\epsilon,\dots, D_m ^{\psi _\epsilon }\psi _\epsilon \right)  \|_{\mathcal{L}^\infty (\mcal O ')}$ uniformly in $\epsilon $ for any $\mcal O ' \Subset \mcal O$. Moreover the uniform convergence of $\left(D_2 ^{\psi _\epsilon }\psi _\epsilon,\dots, D_m ^{\psi _\epsilon }\psi _\epsilon \right) $ allows the choice of a modulus of continuity for $D^{\psi _\epsilon }_j \psi _\epsilon$ which is indipendent of $\epsilon $ for all $j$. Therefore there is $\alpha :(0,+\infty )\to [0,+\infty )$, not depending on $\epsilon $, such that  \eqref{limitealpha} and \eqref{betafin} follow.

  We are going to prove that for each point of $\mcal O'$ there are sufficiently small rectangular neigh\-borhoods $\mathcal I \Subset  \mathcal I'\Subset \mcal O $ and a function  $\alpha :(0,+\infty )\to [0,+\infty )$  such that $\lim_{r\to 0}\alpha (r)=0$ and 
\begin{equation}\label{alpha88}
\sup \biggl\{  \frac{|\psi (A)-\psi  (A')|}{|A-A'|^{1/2} _{\R^{m+n-1}} } : \, A,A'\in \mathcal I, 0< |A-A'|_{\R^{m+n-1}}  \leq r \biggl\} \leq \alpha (r).
\end{equation}
By a standard covering argument the general statement follows.


Precisely we are going to prove \eqref{alpha88} with $\alpha$ defined as 
\begin{equation}\label{defialfa0}
\alpha (r):= \frac{3 \left(1+h \right) }{\mathcal{B}_m} \delta (\max\{1, h^2\} r)+Nr^{1/2}
\end{equation}
where if we put  $$K:= \sup_{A=(x,y)\in \mathcal I'} \sum_{i=2}^{m}|x_i|,\quad M:=\|\psi \|_{\mathcal{L}^\infty (\mathcal I')}, \quad N:=\| w \|_{\mathcal{L}^\infty (\mathcal I')}$$ and  $\beta :(0,+\infty ) \to [0,+\infty )$, increasing,  such that $\lim_{r\to 0^+}\beta (r)=0$ and
\begin{equation*}
|w (A)-w (A')| \leq \beta (\vert A-A'\vert _{\R^{m+n-1}} )\qquad \text { for all  $A,A' \in \mathcal I'$}
\end{equation*}
then $h:=\left(n\mathcal{B}_M ( K+M) \right)^{1/2}$, $\mathcal{B}_m=\min \{  |b_{il}^{s}| : i,l=1,\dots ,m,\, s=1,\dots , n \mbox { and } b_{il}^{s} \ne 0 \}$, $\mathcal{B}_M=\max \{ b_{il}^{s} : i,l=1,\dots ,m \mbox { and }  s=1,\dots , n  \}$ and 
$$\delta (r):=\max \{ r^{1/4}; (\mathcal{B}_M \beta(Er^{1/4}))^{1/2} \}.$$
Here $E>0$ is a constant such that $  |y-y'|_{\R^{n}} + \mathcal{B}_M(K+2M) |y-y'|^{1/4}_{\R^{n}}  \leq E|y-y'|^{1/4}_{\R^{n}}  $ for all $(x,y), (x',y') \in \mathcal I'$. 

We split the proof in several steps.

$\mathbf{Step \, 1.}$
By standard considerations on ordinary differential equations, we know that for each point of $\mcal O'$ there are $r_0 >0$ and rectangular neighborhoods $\mathcal I \Subset  \mathcal I'\Subset \mcal O $ such that for all $A=(x,y)\in \mathcal I $ there is a unique solution $\gamma _A^j\in \C^1([x_j-r_0 ,x_j+r_0],\mathcal I')$ 
of the Cauchy problem
\begin{equation*}
\left\{
\begin{array}{l}
\dot \gamma_A ^j(t)= D^\psi _j(\gamma _A^j(t))= X_j \psi (\gamma _A^j(t))+\psi ( \gamma_A ^j(t)) \, \sum_{s=1}^{n} b_ {j 1}^{s} Y_s \psi (\gamma _A^j(t))\\
\gamma_A ^j(x_j)= A.
\end{array}
\right.
\end{equation*}
More precisely, 
\begin{equation}\label{5.21ok}
\begin{split}
&\gamma ^j_A(t)=\left(x_2,\dots ,x_{j-1}, t,x_{j+1},\dots ,x_m, y_{1,A}^{j}(t),\dots , y_{n,A}^{j}(t)\right), \qquad \text{where}\\
&y_{s,A}^{j}(t) =y_s+\frac{1}{2} (t-x_j)\sum_{i=2}^{m}x_i b_{ji}^{s}+b_{j1}^{s}\int_{x_j}^t \psi (\gamma _A^j (r)) \, dr, \qquad \mbox { for  } s=1,\dots ,n.
\end{split}
\end{equation}
Moreover observe that
\begin{equation}\label{5.21}
\frac{d^2}{dt^2}  y_{s,A}^{j}(t)=\frac{d}{dt} \Big[ \, \frac{1}{2} \sum_{i=2}^{m}x_i b_{ji}^{s}+b_{j1}^{s} \psi (\gamma _A^j (t))\, \Big] = b_{j1}^{s} w_j(\gamma ^j_A(t)).
\end{equation}

$\mathbf{Step \, 2.}$ 
Assume $A,  B\in \mathcal I $ (up to reducing $\mathcal I$) with $A=(x,y)$ and $B=(x,\bar y) $.
We prove that
\begin{equation}\label{finaleCC}
\frac{|\psi (A)-\psi  (B)|}{|y-\bar y|^{1/2} _{\R^{n}} }  \leq  \frac{3}{\mathcal{B}_m}\delta ,
\end{equation}   
where $\delta := \delta(|y-\bar y|_{\R^{n}} )$. Suppose on the contrary that \eqref{finaleCC} is not true, i.e.
\[
\frac{|\psi (A)-\psi  (B)|}{|y-\bar y|^{1/2}_{\R^{n}} }  > \frac{3}{\mathcal{B}_m}\delta .
\]  
Let $b_{ j1}^{s} \ne 0$ for some  $s\in \{1,\dots , n\} $ and $j=2,\dots , m$. Observe that if $b_{ j1}^{s}= 0$ for all $s\in \{1,\dots , n\} $ and $j=2,\dots , m$, then $w_j$ is smooth for each $j=2,\dots , m$.

Let $\gamma ^j_{A}$ and $\gamma ^j_{ B}$ with $$\gamma ^j_{A}(t)=(x_2,\dots, x_{j-1},t,x_{j+1},\dots ,x_m,y_1(t),\dots ,y_n(t))$$ and
$$\gamma ^j_{ B}(t)=(x_2,\dots ,x_{j-1},t,x_{j+1},\dots ,x_m,\bar y_1(t),\dots ,\bar y_n(t)).$$ 
Suppose that $y_s \geq \bar y_s$ (for the other case it is sufficient to exchange the roles of $A$ and  $B$).  By $\eqref{5.21ok}$ and $\eqref{5.21}$, for $t\in [x_j-r_0 ,x_j+r_0 ]$ we have
\[
\begin{aligned}
&y_s(t)-\bar y_s(t) - (y_s-\bar y_s)\\
&=\int _{x_j}^t \biggl[ \dot y_s(x_j) -\dot {\bar y}_s (x_j)+\int _{x_j}^{r'} (\ddot {\bar y }_s (r)-\ddot y'_s (r))\, dr   \biggl] dr'  \\
& = b_{j1}^{s} (t-x_j)(\psi (A)-\psi (B))+ b_{j1}^{s} \int _{x_j}^t\int _{x_j}^{r'} \big( w_j (\gamma ^j_{A}(r)) -w_j (\gamma ^j_{ B}(r))\big)\, dr  dr' .\\
\end{aligned}
\]
Now using the following facts
\begin{equation*}
\begin{aligned}
\max_{ r } | \dot y_s (r)|  = \max_{ r } \big| \frac{1}{2} \sum_{i=1}^{m} x_i b_{ji}^{s} + b_{j1}^{s} \psi (\gamma ^j_{A}( r)) \big|  \leq \mathcal{B}_M \biggl(\frac{1}{2}K+M \biggl)
\end{aligned}
\end{equation*} 
and
\begin{equation*}
\begin{aligned}
|\gamma ^j_{A}(r) - \gamma ^j_{ B}(r)| & \leq |\gamma ^j_{A} (x_j) - \gamma ^j_{ B}(x_j)| + |r-x_j|(\max_{r } | \dot y_s ( r)| + \max_{ r } | \dot {\bar y}_s ( r)| )\\
& \leq \vert A-B\vert _{\R^{m+n-1}}  + |t-x_j|(\max_{ r } | \dot y_s ( r)| + \max_{\hat r } | \dot {\bar y}_s ( r)| )\\
& \leq  |y-\bar y|_{\R^{n}}  +|t-x_j| \mathcal{B}_M (K+2M )
\end{aligned}
\end{equation*} 
we obtain
\begin{equation}\label{5.24}
\begin{aligned}
&y_s(t)-\bar y_s(t) - (y_s-\bar y_s)\\
& \leq  b_{j1}^{s} (t-x_j)(\psi (A)-\psi (B))+|b_{j1}^{s} |(t-x_j)^2 \sup_{r}\beta \big(\vert\gamma ^j_{A} (r) - \gamma ^j_{ B}(r)\vert\big)  \\
& \leq b_{j1}^{s} (t-x_j)(\psi (A)-\psi (B))+
|b_{j1}^{s} |(t-x_j)^2 \beta \left( |y-\bar y| _{\R^{n}} +|t-x_j| \mathcal{B}_M (K+2M ) \right) 
\end{aligned}
\end{equation} 

So if $b_{j1}^{s} (\psi (A)-\psi (B))>0$  put $t:= x_j- \frac{|y-\bar y|_{\R^{n}} ^{1/2}}{\delta }$ in $\eqref{5.24}$ and $t:=x_j+\frac{|y-\bar y|^{1/2}_{\R^{n}} }{\delta }$ otherwise. Observe that in both cases we conclude that
\begin{equation}\label{phi1}
-| b_{j1}^{s} | |\psi (A)-\psi (B)|  < -3 \delta  |y-\bar y|^{1/2}_{\R^{n}} 
\end{equation} 

If $|y-\bar y|_{\R^{n}}$ is "sufficiently small" $\gamma ^j_{A} $ and $\gamma ^j_{B}$ are well defined; it is sufficient to take $r_0 \geq |y-\bar y|_{\R^{n}}^{1/4} \geq |y-\bar y|_{\R^{n}}^{1/2}/\delta  =|t-x_j| $ (it is exactly here that we reduce $\mathcal I$). Using $\eqref{5.24}$, $\eqref{phi1}$ and  the definition of $\beta$ we obtain in both cases 
\begin{equation}\label{5.24bis}
\begin{aligned}
y_s(t)-\bar y_s(t) & \leq y_s-\bar y_s +|b_{j1}^{s}| |y-\bar y|^{1/2}_{\R^{n}} \frac{-|\psi (A)-\psi (B)|}{\delta } +\\
& \hspace{0,5 cm} + \frac{1}{\delta ^2}|b_{j1}^{s}| |y-\bar y| _{\R^{n}} \beta  \left( |y-\bar y|_{\R^{n}} + \mathcal{B}_M(K+2M) \frac{|y-\bar y|^{1/2} _{\R^{n}}}{\delta} \right)    \\
& \leq y_s-\bar y_s  - 3 |y-\bar y|_{\R^{n}}^{1/2} |y-\bar y|^{1/2} _{\R^{n}} + |b_{j1}^{s}| |y-\bar y| _{\R^{n}} \frac{ \beta  \left( E|y-\bar y|^{1/4}_{\R^{n}} \right)}{\delta ^2}    \\
& \leq y_s-\bar y_s - 3 |y-\bar y|_{\R^{n}}+\mathcal{B}_M \biggl( \frac{1}{\mathcal{B}_M}\biggl) |y-\bar y|_{\R^{n}} \leq -|y-\bar y| _{\R^{n}}<0.
\end{aligned}
\end{equation}

This leads to a contradiction; indeed if $y_s>\bar y_s$, then $y_s(\cdot )$ and  $\bar y_s(\cdot )$ are solutions of the same Cauchy problem
\begin{equation}\label{sCp}
\dot y _s(r)=\frac{1}{2} \sum_{i=2}^{m}x_i b_{ji}^{s}+b_{j1}^{s} \psi (x_2,\dots ,x_{j-1}, r,x_{j+1},\dots ,x_m, y_1,\dots ,y_{s-1}, y_s(r) ,y_{s+1}, \dots , y_n )
\end{equation} 
with initial data $y_s(x_j)=y_s$ and $\bar y_s$ respectively, but two such solutions cannot meet, while $y_s(x_j)-\bar y_s(x_j)>0$ and $y_s(t)-\bar y_s(t)<0$  for a certain $t\in (x_j-r_0 , x_j+r_0)$ with $r_0$ sufficiently large. 

On the other hand if $y_s=\bar y_s$, by \eqref{5.24bis} we conclude that $y_s(t) \ne \bar y_s(t)$ for $t=x_j+ \frac{|y-\bar y|^{1/2} _{\R^{n}}}{\delta }$ or $t=x_j- \frac{|y-\bar y|^{1/2}_{\R^{n}}} {\delta }$.  Then we have the contradiction because $y_s(\cdot )$ and  $\bar y_s(\cdot )$ are solutions of the same Cauchy problem \eqref{sCp} with initial data $y_s(t)$ and $ \bar y_s(t)$ but two such solutions cannot meet, while $y_s=\bar y_s$.

Hence \eqref{finaleCC} follows.

$\mathbf{Step \, 3.}$ 
Now let $A, A', B\in \mathcal I $ (up to reducing $\mathcal I$) with $A=(x,y)$, $A'=(x',y') $ and $B=(x,y') $. We want to show that 
\begin{equation}\label{equVERIFICA}
\frac{\vert \psi(B)- \psi (A')\vert }{\vert x-x'\vert^{1/2} _{\R^{m-1}}} \leq N |A-A'|^{1/2}_{\R^{m+n-1}} + \frac{3h \delta }{\mathcal{B}_m}
\end{equation}
where $\delta = \delta( h^2 \vert x-x'\vert _{\R^{m-1}})$. We suppose on the contrary that
\begin{equation}\label{equVERIFICA22}
\frac{\vert \psi(B)- \psi (A')\vert }{\vert x-x'\vert^{1/2} _{\R^{m-1}} } > N |A-A'|^{1/2}_{\R^{m+n-1}} + \frac{3h \delta }{\mathcal{B}_m}. 
\end{equation}
Set 
\begin{equation*}
\begin{aligned}
D_{j} & :=\gamma ^j_{D_{j-1}}(x_{j}) \quad \mbox{ for } j=2,\dots ,m \\
\end{aligned}
\end{equation*}
with $D_1:=A'$. A computation gives that 
\begin{equation*}
D_j=(x_2,\dots ,x_j,x'_{j+1},\dots ,x'_m,y^{D_j})
\end{equation*} 
with
\begin{equation*}
\begin{aligned}
y^{D_j}_s & = y'_s+\sum_{l=2}^{j} \left( b_{l1}^{s} \int_{x'_l}^{x_l} \psi \left(\gamma ^j_{D_{j-1}}(r)  \right)\, dr +\frac{1}{2}(x_l-x'_l)\Big(\sum_{ i=2 }^{l} x_i b_{li}^{s}+\sum_{ i=l+1}^{m} x'_i b_{li}^{s}\Big)  \right)\\
\end{aligned}
\end{equation*} 
for $s=1,\dots ,n$ and consequently, recalling that $h=\left(n\mathcal{B}_M ( K+M) \right)^{1/2}$
\begin{equation}\label{rapportoD_m}
\left| y' - y^{D_m}\right| _{\R^{n}} \leq n \mathcal{B}_M \left(M + K \right)|x-x'| _{\R^{m-1}}  = h^{2} |x-x'|_{\R^{m-1}} .
\end{equation} 
 Moreover we have 
\begin{equation*}
 \sum_{j=2}^m \left|\psi (D_{j-1})-\psi (D_j) \right|  =  \sum_{j=2}^m \left|\int_{x'_j}^{x_j} w_j (\gamma ^j_{D_{j-1}}(t))\, dt \right|   \leq N|x-x'|_{\R^{m-1}} 
\end{equation*} 
In order to stay in the $\mathcal I $ of the previous step, we want $x'-x$ sufficiently small (and precisely when $N|x-x'|^{1/2}_{\R^{m-1}}  \leq |x-x'|^{1/4} _{\R^{m-1}}  \leq \delta $). By \eqref{equVERIFICA22} and \eqref{rapportoD_m}, we get
\begin{equation*}
\begin{aligned}
|\psi (B)-\psi (D_m)| & \geq |\psi (B)-\psi (A')|- \sum_{l=2}^m|\psi (D_{l-1})-\psi (D_l)| \\
 & > \left(N |A-A'|^{1/2} _{\R^{m+n-1}} + \frac{3 h }{\mathcal{B}_m}  \delta -N|x-x'|_{\R^{m-1}} ^{1/2}\right) |x-x'|^{1/2}_{\R^{m-1}}  \\
 & \geq  \frac{3h }{\mathcal{B}_m} \delta |x-x'|^{1/2}_{\R^{m-1}} \\
 & \geq \frac{3}{\mathcal{B}_m}\delta  |y'-y^{D_m}|^{1/2}_{\R^{n}} 
\end{aligned}
\end{equation*} 
so that we are in the first case again (see $\eqref{finaleCC}$ with the couple $B=(x,y'),D_m=(x,y^{D_m})$ instead $A,B$ respectively) which we have seen is not possible. Hence \eqref{equVERIFICA} holds.

$\mathbf{Step \, 4.}$ 
Using Step 2. and Step 3. we deduce that 
\begin{equation*}
\begin{aligned}
 \frac{|\psi (A)-\psi  (A')|}{|A-A'|^{1/2} _{\R^{m+n-1}}  } & \leq \frac{|\psi (A)-\psi  (B)|}{|y-y'|^{1/2} _{\R^{n}} } + \frac{|\psi (B)-\psi  (A')|}{|x-x'|_{\R^{m-1}} ^{1/2}} \\
 & \leq \frac{3 \left(1+h \right) }{\mathcal{B}_m} \delta ( \max\{1, h^2\} |A-A'|_{\R^{m+n-1}} ) +N|A-A'|^{1/2}_{\R^{m+n-1}}  \\
 &=\alpha (|A-A'|_{\R^{m+n-1}} )
\end{aligned}
\end{equation*}
for all $A=(x,y), A' =(x',y'), B=(x,y') \in \mathcal I $.  Then according to \eqref{alpha88} and \eqref{defialfa0} we have that $\lim_{r\to 0}\alpha (r)=0$ and we are able to control $\alpha $ with only $K,M,N,\mathcal{B}_M, \mathcal{B}_m$ and $\beta $.
%
%

\end{proof}

\end{document}